\documentclass[12pt,a4paper,reqno]{amsart}
\usepackage[headings]{fullpage}
\usepackage{amsthm}
\usepackage{amssymb}
\usepackage[nobysame, alphabetic]{amsrefs}
\usepackage[all]{xy}
\UseComputerModernTips
\numberwithin{equation}{section}
\theoremstyle{plain}
\newtheorem{theorem}{Theorem}
\numberwithin{theorem}{section}
\newtheorem{proposition}[theorem]{Proposition}
\newtheorem{lemma}[theorem]{Lemma}
\newtheorem{corollary}[theorem]{Corollary}
\theoremstyle{definition}
\newtheorem{definition}[theorem]{Definition}

\theoremstyle{remark}
\newtheorem{remark}[theorem]{Remark}
\newtheorem*{acknowledgements}{Acknowledgments}
\newtheorem*{notations}{Notations and conventions}
\newcommand{\F}{{\mathbb F}}
\newcommand{\Z}{{\mathbb Z}}
\newcommand{\Q}{{\mathbb Q}}
\renewcommand{\H}{{\mathbb H}}
\newcommand{\ff}{{\text{\it ff\/}}}
\newcommand{\thh}{{T\kern -1pt H\kern -1pt H}}
\newcommand{\xr}{\xrightarrow}
\newcommand{\xl}{\xleftarrow}
\newcommand{\cok}{\operatorname{cok}}
\newcommand{\im}{\operatorname{im}}
\newcommand{\Ext}{\operatorname{Ext}}
\newcommand{\holim}{\operatornamewithlimits{holim}}
\newcommand{\Rlim}{\operatornamewithlimits{lim^1}}
\newcommand{\ADer}{\operatorname{ADer}}
\newcommand{\dlog}{\operatorname{dlog}}
\renewcommand{\:}{\colon}
\begin{document}
\title{Algebraic $K$-theory of the first Morava $K$-theory}
\author{Christian Ausoni}
\address{Mathematical Institute, University of M\"unster,
DE--48149 M\"unster, Germany}
\email{ausoni@uni-muenster.de}
\author{John Rognes} 
\address{Department of Mathematics, University of Oslo,
NO--0316 Oslo, Norway}
\email{rognes@math.uio.no}
\subjclass[2000]{19D55, 55N15}
\begin{abstract} 
For a prime $p \ge 5$, we compute the algebraic $K$-theory modulo $p$ and $v_1$ of the
mod $p$ Adams summand, using topological cyclic homology.
On the way, we evaluate its modulo
$p$ and $v_1$ topological Hochschild homology.
Using a localization sequence, we also compute the
$K$-theory modulo $p$ and $v_1$ of the first Morava
$K$-theory. 
\end{abstract}
\maketitle
\section{Introduction }

In this paper we continue the investigation from \cite{kkmodp-AR02} and
\cite{kkmodp-Au10} of the algebraic $K$-theory of topological $K$-theory
and related $S$-algebras.  Let $\ell_p$ be the $p$-complete Adams
summand of connective complex $K$-theory, and let $\ell/p = k(1)$ be
the first connective Morava $K$-theory.  It has a unique $S$-algebra
structure~\cite{kkmodp-Ang}*{Th.~A}, and we show in Section~\ref{kkmodp-sec:base}
that $\ell/p$ is an $\ell_p$-algebra (in uncountably many ways), so that
$K(\ell/p)$ is a $K(\ell_p)$-module spectrum.

Let $V(1) = S/(p,v_1)$ be the type~$2$ Smith--Toda complex
(see Section~\ref{kkmodp-sec:v1}
below for a definition).  It is
a homotopy commutative ring spectrum for $p\ge5$, with a preferred
periodic class $v_2 \in V(1)_*$ of degree $2p^2-2$.  
We write $V(1)_*(X) = \pi_*(V(1)
\wedge X)$ for the $V(1)$-homotopy of a spectrum~$X$.  Multiplication by
$v_2$ makes $V(1)_*(X)$ a $P(v_2)$-module, where $P(v_2)$ denotes the
polynomial algebra over~$\F_p$ generated by $v_2$.
We denote by $\F_p\{x_1,\dots, x_n\}$
the $\F_p$-vector space generated by $x_1,\dots,x_n$, and
by $E(x_1,\dots, x_n)$ the exterior algebra over $\F_p$ generated 
by $x_1,\dots,x_n$. 

We computed the $V(1)$-homotopy of $K(\ell_p)$
in~\cite{kkmodp-AR02}*{Th.~9.1}, showing
that it is essentially a free $P(v_2)$-module on $(4p+4)$ generators.
The following is our main result, corresponding
to Theorem~\ref{kkmodp-thm:8.10} in the body of the paper.

\begin{theorem}\label{kkmodp-thm:1.2}
Let $p\ge5$ be a prime and let $\ell/p = k(1)$ be the first connective
Morava $K$-theory spectrum.
There is an isomorphism of $P(v_2)$-modules
\begin{equation*}
\begin{aligned}
V(1)_* K(\ell/p) &\cong P(v_2) \otimes E(\bar\epsilon_1) \otimes
	\F_p\{1, \partial\lambda_2, \lambda_2, \partial v_2\} \\
&\qquad \oplus P(v_2) \otimes E(\dlog v_1) \otimes \F_p\{t^d
	v_2 \mid 0<d<p^2-p, p\nmid d\} \\
&\qquad \oplus P(v_2) \otimes E(\bar\epsilon_1) \otimes \F_p\{t^{dp}
	\lambda_2 \mid 0<d<p\} \,.
\end{aligned}
\end{equation*} 
Here $|\lambda_1| = |\bar\epsilon_1| = 2p-1$,
$|\lambda_2| = 2p^2-1$, $|v_2| = 2p^2-2$, $|\dlog v_1| = 1$, $|\partial| =
-1$ and $|t| = -2$.  This is a free $P(v_2)$-module of rank~$(2p^2-2p+8)$
and of zero Euler characteristic.
\end{theorem}

We prove this theorem by means of the cyclotomic trace map \cite{kkmodp-BHM93}
to topological cyclic homology $TC(\ell/p;p)$.  Along the way we evaluate
$V(1)_* \thh(\ell/p)$, where $\thh$ denotes topological Hochschild
homology, as well as $V(1)_* TC(\ell/p;p)$, see Proposition~\ref{kkmodp-prop:5.4}
and Theorem~\ref{kkmodp-thm:8.8}.

Let $L_p$ be the $p$-complete Adams summand of periodic complex
$K$-theory, and let $L/p = K(1)$ be the first periodic Morava
$K$-theory.  The localization cofiber sequence $K(\Z_p) \to
K(\ell_p) \to K(L_p) \to \Sigma K(\Z_p)$
of Blumberg and Mandell~\cite{kkmodp-BM08}*{p.~157} has the mod~$p$ Adams analogue
\begin{equation*}
K(\Z/p) \to K(\ell/p) \to K(L/p) \to \Sigma K(\Z/p) \,,
\end{equation*}
see Proposition~\ref{kkmodp-prop:2.99} below.
Using Quillen's computation \cite{kkmodp-Qu72}*{Th.~7} of $K(\Z/p)$, we obtain the following
consequence:

\begin{corollary}
Let $p\ge5$ be a prime and let $L/p = K(1)$ be the first Morava
$K$-theory spectrum.
There is an isomorphism of $P(v_2^{\pm1})$-modules
\begin{equation*}  
  V(1)_*K(L/p)[v_2^{-1}]
  \cong
  V(1)_*K(\ell/p)[v_2^{-1}]
  \,.
\end{equation*}
If there is a class $\dlog v_1\in V(1)_1 K(L/p)$ with $\lambda_2 = v_2
\cdot \dlog v_1$, then there is an isomorphism of $P(v_2)$-modules
\begin{equation*}
\begin{aligned}
V(1)_*K(L/p)  &\cong P(v_2) \otimes E(\bar\epsilon_1)\otimes 
	\F_p\{1, \partial\lambda_2, \dlog v_1, \partial v_2\} \\
&\qquad \oplus P(v_2) \otimes E(\dlog v_1) \otimes
\F_p\{t^d v_2 \mid 0<d<p^2-p, p\nmid d\} \\
&\qquad \oplus P(v_2) \otimes E(\bar\epsilon_1) \otimes
	\F_p\{t^{dp} v_2 \dlog v_1 \mid 0<d<p\} \,,
\end{aligned}
\end{equation*}
where the degrees of the generators are as in Theorem~\ref{kkmodp-thm:1.2}.
This is a free $P(v_2)$-module of rank~$(2p^2-2p+8)$ and of zero Euler
characteristic.
\end{corollary}

Our far-reaching aim, which partially motivated the computations
presented here, is to conceptually understand the algebraic
$K$-theory of $\ell_p$ and other commutative $S$-algebras in terms of
localization and Galois descent, in the same way as we understand the
algebraic $K$-theory of rings of integers in (local) number fields or
more general regular rings.  The first task is to relate $K(\ell_p)$ to
the algebraic $K$-theory of its ``residue fields'' and ``fraction field'',
for which we expect a description in terms of Galois cohomology to exist,
starting with the Galois theory for commutative $S$-algebras developed by
the second author~\cite{kkmodp-Rog08}.  The residue rings of $\ell_p$ appear to
be $\ell/p$, $H\Z_p$ and~$H\Z/p$, while the fraction field
$\ff(\ell_p)$ is more mysterious.
For our purposes, its algebraic $K$-theory $K(\ff(\ell_p))$
should fit in a natural localization cofibre sequence of
spectra
\begin{equation*}
K(L/p) \to K(L_p) \to K(\ff(\ell_p)) \to \Sigma K(L/p)  \,.
\end{equation*}
An obvious candidate for $\ff(\ell_p)$ is provided by the algebraic
localization $L_p[p^{-1}] = L\Q_p$, having as coefficients the
graded field $\Q_p[v_1^{\pm1}]$. However, by the
following corollary, this is too naive.

\begin{corollary}
The spectra $ K(L/p)$, $K(L_p)$ and $K(L\Q_p)$ cannot
possibly fit in a cofibre sequence 
\begin{equation*}
 K(L/p) \to K(L_p) \to K(L\Q_p) \to \Sigma K(L/p)\,.
\end{equation*}
\end{corollary}

Indeed, the above computation implies that
$V(1)_*K(L/p)[v_2^{-1}]$ and $V(1)_*K(L_p)[v_2^{-1}]$ are
not abstractly isomorphic, while
$V(1)_*K(L\Q_p)[v_2^{-1}]$ is zero since it is an algebra
over $V(1)_*K(\Q_p)[v_2^{-1}]=0$. The later equality follows
from the computation of the $p$-primary homotopy type of
$K(\Q_p)$~\cite{kkmodp-HM03}*{Th.~D}, which shows that
$V(1)_*K(\Q_p)$ is $v_2$-torsion. 

In conclusion, the conjectural fraction field $\ff(\ell_p)$
appears to be a localization of $L_p$ away from $L/p$ less drastic
than the algebraic localization $L_p[p^{-1}] = L\Q_p$. 
We elaborate more on this issue in~\cite{kkmodp-ARff}.

\medskip

The paper is organized as follows.  In Section~\ref{kkmodp-sec:base} we fix our
notations, show that $\ell/p$ admits the structure of an associative
$\ell_p$-algebra, and give a similar discussion for $ku/p$ and the periodic
versions $L/p$ and $KU/p$. 
Section~\ref{kkmodp-sec:thh} contains the computation of the mod~$p$
homology of $\thh(\ell/p)$, and in Section~\ref{kkmodp-sec:v1} we evaluate
its $V(1)$-homotopy.  In Section~\ref{kkmodp-sec:tate} we show that the
$C_{p^n}$-fixed points and $C_{p^n}$-homotopy fixed points of
$\thh(\ell/p)$ are closely related, and use this to inductively
determine their $V(1)$-homotopy in Section~\ref{kkmodp-sec:higher}.  Finally,
in Section~\ref{kkmodp-sec:tc} we achieve the computation of
$TC(\ell/p;p)$
and $K(\ell/p)$ in $V(1)$-homotopy.

\begin{notations}
Let $p$ be a fixed prime.
We write $E(x) = \F_p[x]/(x^2)$ for the exterior algebra, $P(x) = \F_p[x]$
for the polynomial algebra and $P(x^{\pm1}) = \F_p[x, x^{-1}]$ for the
Laurent polynomial algebra on one generator $x$, and similarly for a
list of generators.  We will also write $\Gamma(x) = \F_p\{ \gamma_i(x)
\mid i \ge 0\}$ for the divided power algebra, with $\gamma_i(x) \cdot
\gamma_j(x) = (i,j) \gamma_{i+j}(x)$, where $(i,j) = (i+j)!/i!j!$ is the
binomial coefficient.  We use the obvious abbreviations $\gamma_0(x) =
1$ and $\gamma_1(x) = x$.  Finally, we write $P_h(x) = \F_p[x]/(x^h)$ for
the truncated polynomial algebra of height $h$, and recall the isomorphism
$\Gamma(x) \cong P_p(\gamma_{p^e}(x) \mid e\ge0)$ in characteristic~$p$.
We write $X_{(p)}$ and~$X_p$ for the $p$-localization 
and the $p$-completion, respectively, of any spectrum or abelian group~$X$.  
In the spectral sequences (of $\F_p$-modules) discussed
below, we often determine differentials
only up to multiplication by a unit. We use the notation
$d(x)\doteq y$ to indicate that the equation $d(x)=\alpha
y$ holds for some unit $\alpha\in\F_p$. 
\end{notations}

\section{Base change squares of $S$-algebras }
\label{kkmodp-sec:base}

Let $p$ be a prime, even or odd for now.  
Let $ku$ and $KU$
be the connective and the periodic complex $K$-theory spectra, with
homotopy rings $ku_* = \Z[u]$ and $KU_* = \Z[u^{\pm1}]$, where
$|u|=2$.  Let $\ell = BP\langle1\rangle$ and $L = E(1)$ be the
$p$-local Adams summands, with $\ell_* = \Z_{(p)}[v_1]$ and $L_* =
\Z_{(p)}[v_1^{\pm1}]$, where $|v_1| = 2p-2$.  The inclusion $\ell \to
ku_{(p)}$ maps $v_1$ to $u^{p-1}$.  Alternate notations in the
$p$-complete cases are $KU_p = E_1$ and $L_p = \widehat{E(1)}$.  These
ring spectra are all commutative $S$-algebras, in the sense that each
admits a unique $E_\infty$ ring spectrum structure.  See
\cite{kkmodp-BaR05}*{p.~692}
for proofs of uniqueness in the periodic cases.

Let $ku/p$ and $KU/p$ be the connective and periodic mod~$p$ complex
$K$-theory spectra, with coefficients $(ku/p)_* = \Z/p[u]$ and $(KU/p)_*
= \Z/p[u^{\pm1}]$.  These are $2$-periodic versions of the first Morava
$K$-theory spectra $\ell/p = k(1)$ and $L/p = K(1)$, with $(\ell/p)_*
= \Z/p[v_1]$ and $(L/p)_* = \Z/p[v_1^{\pm1}]$.  Each of these can be
constructed as the cofiber of the multiplication by $p$ map, as a module
over the corresponding commutative $S$-algebra.  For example, there is
a cofiber sequence of $ku$-modules $ku \xr{p} ku \xr{i} ku/p
\to \Sigma ku$.

Let $HR$ be the Eilenberg--Mac\,Lane spectrum of a ring $R$.  When $R$
is associative, $HR$ admits a unique associative $S$-algebra structure,
and when $R$ is commutative, $HR$ admits a unique commutative
$S$-algebra structure.  The zeroth Postnikov section defines unique
maps of commutative $S$-algebras $\pi \: ku \to H\Z$ and $\pi \: \ell
\to H\Z_{(p)}$, which can be followed by unique commutative $S$-algebra
maps to $H\Z/p$.

The $ku$-module spectrum $ku/p$ does not admit the structure of a
commutative $ku$-algebra.  It cannot even be an $E_2$ or $H_2$ ring
spectrum, since the homomorphism induced in mod~$p$ homology by the
resulting map $\pi \: ku/p \to H\Z/p$ of $H_2$ ring spectra would not
commute with the homology operation $Q^1(\bar\tau_0) = \bar\tau_1$ in
the target $H_*(H\Z/p; \F_p)$ \cite{kkmodp-BMMS86}*{III.2.3}.  Similar remarks
apply for $KU/p$, $\ell/p$ and $L/p$.  Associative algebra structures,
or $A_\infty$ ring spectrum structures, are easier to come by.  The
following result is a direct application of the methods of
\cite{kkmodp-La01}*{\S\S9--11}.  We adapt the notation of
\cite{kkmodp-BJ02}*{\S3} to provide some
details in our case.

\begin{proposition}
The $ku$-module spectrum $ku/p$ admits the structure of an associative
$ku$-algebra, but the structure is not unique.  Similar statements hold
for $KU/p$ as a $KU$-algebra, $\ell/p$ as an $\ell$-algebra and $L/p$
as an $L$-algebra.
\end{proposition}

\begin{proof}
We construct $ku/p$ as the (homotopy) limit of its Postnikov tower of
associative $ku$-algebras $P^{2m-2} = ku/(p, u^m)$, with coefficient rings
$ku/(p, u^m)_* = ku_*/(p, u^m)$ for $m\ge1$.  To start the induction,
$P^0 = H\Z/p$ is a $ku$-algebra via $i \circ \pi \: ku \to H\Z \to H\Z/p$.
Assume inductively for $m\ge1$ that $P = P^{2m-2}$ has been constructed.
We will define $P^{2m}$ by a (homotopy) pullback diagram
\begin{equation*}
\xymatrix{
P^{2m} \ar[r] \ar[d] & P \ar[d]^{in_1} \\
P \ar[r]^-{d} & P \vee \Sigma^{2m+1} H\Z/p
}
\end{equation*}
in the category of associative $ku$-algebras.  Here
\begin{equation*}
d \in \ADer_{ku}^{2m+1}(P, H\Z/p)
\cong \thh_{ku}^{2m+2}(P, H\Z/p)
\end{equation*}
is an associative $ku$-algebra derivation of $P$ with values in
$\Sigma^{2m+1} H\Z/p$, and the group of such can be identified with the
indicated topological Hochschild cohomology group of $P$ over
$ku$.  We recall that these are the homotopy groups (cohomologically
graded) of the function spectrum $F_{P \wedge_{ku} P^{op}}(P, H\Z/p)$.
The composite map $pr_2 \circ d \: P \to \Sigma^{2m+1}
H\Z/p$ of $ku$-modules, where $pr_2$ projects onto the second wedge
summand, is restricted to equal the $ku$-module Postnikov $k$-invariant
of $ku/p$ in
\begin{equation*}
H^{2m+1}_{ku}(P; \Z/p) = \pi_0 F_{ku}(P, \Sigma^{2m+1} H\Z/p) \,.
\end{equation*}
We compute that $\pi_*(P \wedge_{ku} P^{op}) = ku_*/(p, u^m) \otimes
E(\tau_0, \tau_{1,m})$, where $|\tau_0| = 1$, $|\tau_{1,m}| = 2m+1$ and
$E(-)$ denotes the exterior algebra on the given generators.  (For $p=2$,
the use of the opposite product is essential here
\cite{kkmodp-An08}*{\S3}.)
The function spectrum description of topological Hochschild cohomology
leads to the spectral sequence
\begin{equation*}
\begin{aligned}
E_2^{*,*} &= \Ext^{*,*}_{\pi_*(P \wedge_{ku} P^{op})}(\pi_*(P), \Z/p) \\
&\cong \Z/p[y_0, y_{1,m}] \\
&\Longrightarrow \thh^*_{ku}(P, H\Z/p) \,,
\end{aligned}
\end{equation*}
where $y_0$ and $y_{1,m}$ have cohomological bidegrees~$(1,1)$
and~$(1,2m+1)$, respectively.  The spectral sequence collapses at $E_2 =
E_\infty$, since it is concentrated in even total degrees.  In particular,
\begin{equation*}
\ADer^{2m+1}_{ku}(P, H\Z/p) \cong \F_p\{y_{1,m}, y_0^{m+1}\} \,.
\end{equation*}
Additively, $H^{2m+1}_{ku}(P; \Z/p) \cong \F_p\{Q_{1,m}\}$ is generated
by a class dual to $\tau_{1,m}$, which is the image of $y_{1,m}$ under
left composition with $pr_2$.  It equals the $ku$-module $k$-invariant
of $ku/p$.  Thus there are precisely $p$ choices $d = y_{1,m} +
\alpha y_0^{m+1}$, with $\alpha \in \F_p$, for how to extend any given
associative $ku$-algebra structure on $P = P^{2m-2}$ to one on $P^{2m}
= ku/(p, u^{m+1})$.  In the limit, we find that there are an uncountable
number of associative $ku$-algebra structures on $ku/p = \holim_m P^{2m}$,
each indexed by a sequence of choices $\alpha \in \F_p$ for all $m\ge1$.

The periodic spectrum $KU/p$ can be obtained from $ku/p$ by Bousfield
$KU$-lo\-ca\-li\-za\-tion in the category of $ku$-modules
\cite{kkmodp-EKMM97}*{VIII.4}, which makes it an associative $KU$-algebra.  The
classification of periodic $S$-algebra structures is the same as in the connective
case, since the original $ku$-algebra structure on $ku/p$ can be
recovered from that on $KU/p$ by a functorial passage to the connective
cover.  To construct $\ell/p$ as an associative $\ell$-algebra, or
$L/p$ as an associative $L$-algebra, replace $u$ by $v_1$ in these
arguments.
\end{proof}

By varying the ground $S$-algebra, we obtain the same conclusions about
$ku/p$ as a $ku_{(p)}$-algebra or $ku_p$-algebra, and about $\ell/p$ as
an $\ell_p$-algebra.  

For each choice of $ku$-algebra structure on $ku/p$, the zeroth
Postnikov section 
\begin{equation*}
\pi \: ku/p \to H\Z/p
\end{equation*}
is a $ku$-algebra map, with
the unique $ku$-algebra structure on the target.  Hence there is
a commutative square of associative $ku$-algebras
\begin{equation*}
\xymatrix{
ku \ar[r]^i \ar[d]^{\pi} & ku/p \ar[d]^{\pi} \\
H\Z \ar[r]^i & H\Z/p
}
\end{equation*}
and similarly in the $p$-local and $p$-complete cases.  In view of the
weak equivalence $H\Z \wedge_{ku} ku/p \simeq H\Z/p$, this square
expresses the associative $H\Z$-algebra $H\Z/p$ as the base change of
the associative $ku$-algebra $ku/p$ along $\pi \: ku \to H\Z$.
Likewise, there is a commutative square of associative $\ell_p$-algebras
\begin{equation}
\xymatrix{
\ell_p \ar[r]^i \ar[d]^{\pi} & \ell/p \ar[d]^{\pi} \\
H\Z_p \ar[r]^i & H\Z/p
}
\label{kkmodp-eq:2.4}
\end{equation}
that expresses $H\Z/p$ as the base change of $\ell/p$ along $\ell_p \to
H\Z_p$, and similarly in the $p$-local case.  By omission of
structure, these squares are also diagrams of $S$-algebras and
$S$-algebra maps.

We end this section by formulating the mod $p$ analogue of the 
localization cofibre sequence in algebraic $K$-theory 
\begin{equation}
K(\Z_p) \to K(\ell_p) \to K(L_p) \to \Sigma K(\Z_p)
\label{kkmodp-eq:2.98}
\end{equation}
conjectured by the second author and established by 
Blumberg and Mandell~\cite{kkmodp-BM08}*{p.~157}

\begin{proposition}\label{kkmodp-prop:2.99}
There is a localization cofibre sequence of spectra
\begin{equation*}
K(\Z/p) \to K(\ell/p) \to K(L/p) \to \Sigma K(\Z/p)
\end{equation*}
where the first map is the transfer and the second map is
induced by the localization $\ell/p\to\ell/p[v_1^{-1}]=L/p$.
\end{proposition}
\begin{proof}
The proof of the existence of the localization
sequence~\eqref{kkmodp-eq:2.98} given in~\cite{kkmodp-BM08}*{p.
160--163}
and the identification of the transfer map
adapt without change to cover the mod $p$ analogue stated
in this proposition. Here we use that a finite cell
$\ell/p$-module that is $v_1$-torsion has finite homotopy
groups, and the non-zero groups are concentrated in a finite range of degrees.
\end{proof}

\section{Topological Hochschild homology }
\label{kkmodp-sec:thh}
We shall compute the $V(1)$-homotopy of the topological Hochschild
homology $\thh(-)$ and topological cyclic homology $TC(-;p)$ of the
$S$-algebras in diagram~\eqref{kkmodp-eq:2.4}, for primes $p\ge5$.  Passing to
connective covers, this also computes the $V(1)$-homotopy of the algebraic
$K$-theory spectra appearing in that square.  With these coefficients,
or more generally, after $p$-adic completion, the functors $\thh$ and $TC$
are insensitive to $p$-completion in the argument, so we shall simplify
the notation slightly by working with the associative $S$-algebras $\ell$
and $H\Z_{(p)}$ in place of $\ell_p$ and $H\Z_p$.  For ordinary rings $R$
we almost always shorten notations like $\thh(HR)$ to $\thh(R)$.

The computations follow the strategy of \cite{kkmodp-Bo2}, \cite{kkmodp-BM94},
\cite{kkmodp-BM95} and \cite{kkmodp-HM97} for $H\Z/p$ and $H\Z$, and of \cite{kkmodp-MS93}
and \cite{kkmodp-AR02} for $\ell$.  See also \cite{kkmodp-AnR05}*{\S\S4--7} for further
discussion of the $\thh$-part of such computations.  In this section
we shall compute the mod~$p$ homology of the topological Hochschild
homology of $\ell/p$ as a module over the corresponding homology for
$\ell$, for any odd prime~$p$.

\begin{remark}\label{kkmodp-rem:3.99}
Our computations are based on comparisons,
using the maps displayed in
diagram~\eqref{kkmodp-eq:2.4} above. We will abuse notation and use
the same name for classes in the homology or
$V(1)$-homotopy of $\thh(\ell_p)$, $\thh(\ell/p)$,
$\thh(\Z_p)$ or $\thh(\Z/p)$, when these classes
unambiguously correspond
to each other under the homomorphisms induced by the maps $i$ and
$\pi$ in~\eqref{kkmodp-eq:2.4}. We also use this abuse of notations
in later sections for the $V(1)$-homotopy of
$TC$, etc.
\end{remark}

We write $H_*(-)$ for homology with mod~$p$ coefficients.  It takes
values in graded $A_*$-comodules, where $A_*$ is the dual Steenrod algebra
\cite{kkmodp-Mi58}*{Th~2}.  Explicitly (for $p$ odd),
\begin{equation*}
A_* = P(\bar\xi_k \mid k\ge1) \otimes E(\bar\tau_k \mid k\ge0)
\end{equation*}
with coproduct
\begin{equation*}
\psi(\bar\xi_k) = \sum_{i+j=k} \bar\xi_i \otimes \bar\xi_j^{p^i}
\end{equation*}
and
\begin{equation*}
\psi(\bar\tau_k) = 1 \otimes \bar\tau_k + \sum_{i+j=k}  \bar\tau_i
\otimes \bar\xi_j^{p^i} \,.
\end{equation*}
Here $\bar\xi_0 = 1$, $\bar\xi_k = \chi(\xi_k)$ has degree $2(p^k-1)$
and $\bar\tau_k = \chi(\tau_k)$ has degree $2p^k-1$, where $\chi$ is
the canonical conjugation \cite{kkmodp-MM65}*{8.4}.  Then the maps $i$
and the zeroth
Postnikov sections $\pi$ of~\eqref{kkmodp-eq:2.4} induce identifications
\begin{equation*}
\begin{aligned}
H_*(H\Z_{(p)}) &=  P(\bar\xi_k \mid k\ge1) \otimes E(\bar\tau_k \mid k\ge1) \\
H_*(\ell) &=  P(\bar\xi_k \mid k\ge1) \otimes E(\bar\tau_k \mid k\ge2) \\
H_*(\ell/p) &=  P(\bar\xi_k \mid k\ge1) \otimes E(\bar\tau_0, \bar\tau_k
\mid k\ge2)
\end{aligned}
\end{equation*}
as $A_*$-comodule subalgebras of $H_*(H\Z/p) = A_*$.
We often make use of the following $A_*$-comodule coactions
\begin{equation*}
\begin{aligned}
\nu(\bar\tau_0) &= 1 \otimes \bar\tau_0 + \bar\tau_0 \otimes 1 \\
\nu(\bar\xi_1) &= 1 \otimes \bar\xi_1 + \bar\xi_1 \otimes 1 \\
\nu(\bar\tau_1) &= 1 \otimes \bar\tau_1 + \bar\tau_0 \otimes \bar\xi_1
+ \bar\tau_1 \otimes 1 \\
\nu(\bar\xi_2) &= 1 \otimes \bar\xi_2 + \bar\xi_1 \otimes \bar\xi_1^p +
\bar\xi_2 \otimes 1 \\
\nu(\bar\tau_2) &= 1 \otimes \bar\tau_2 + \bar\tau_0 \otimes \bar\xi_2
+ \bar\tau_1 \otimes \bar\xi_1^p + \bar\tau_2 \otimes 1 \,.
\end{aligned}
\end{equation*}

The B{\"o}kstedt spectral sequences
\begin{equation*}
E^2(B) = HH_*(H_*(B)) \Longrightarrow H_*(\thh(B))
\end{equation*}
for the commutative $S$-algebras $B = H\Z/p$, $H\Z_{(p)}$ and $\ell$
begin
\begin{equation*}
\begin{aligned}
E^2(\Z/p) &= A_* \otimes E(\sigma\bar\xi_k \mid k\ge1) \otimes
\Gamma(\sigma\bar\tau_k \mid k\ge0) \\
E^2(\Z_{(p)}) &= H_*(H\Z_{(p)}) \otimes E(\sigma\bar\xi_k \mid
k\ge1) \otimes \Gamma(\sigma\bar\tau_k \mid k\ge1) \\
E^2(\ell) &= H_*(\ell) \otimes E(\sigma\bar\xi_k \mid k\ge1)
\otimes \Gamma(\sigma\bar\tau_k \mid k\ge2) \,.
\end{aligned}
\end{equation*}
Here $HH_*(H_*(B))$ denotes the Hochschild homology of the
graded $\F_p$-algebra $H_*(B)$. In the above formula we made
use of the $\F_p$-linear operator
$\sigma \: H_*(B)\to HH_1(H_*(B))$, $x\mapsto \sigma x$, where $\sigma x$
is the class represented by $1\otimes x-x\otimes1$ in the Hochschild
complex. Notice that $\sigma$ is the restriction of Connes'
operator $d$ to $HH_0(H_*(B))=H_*(B)$, and is a 
derivation in the sense that 
\begin{equation*}
\sigma(xy)=x\sigma(y)+ (-1)^{|x||y|}y\sigma(x)
\end{equation*}
for all $x,y\in H_*(B)$. 
These spectral sequences are (graded) commutative $A_*$-comodule 
algebra spectral sequences,
and there are differentials
\begin{equation*}
d^{p-1}(\gamma_j\sigma\bar\tau_k) \doteq \sigma\bar\xi_{k+1}
\cdot \gamma_{j-p}\sigma\bar\tau_k
\end{equation*}
for $j\ge p$ and $k\ge0$, see \cite{kkmodp-Bo2}*{Lem.~1.3},
\cite{kkmodp-Hu96}*{Th.~1}
or \cite{kkmodp-Au05}*{Lem.~5.3}, leaving
\begin{equation*}
\begin{aligned}
E^\infty(\Z/p) &= A_* \otimes P_p(\sigma\bar\tau_k \mid k\ge0) \\
E^\infty(\Z_{(p)}) &= H_*(H\Z_{(p)}) \otimes E(\sigma\bar\xi_1)
\otimes P_p(\sigma\bar\tau_k \mid k\ge1) \\
E^\infty(\ell) &= H_*(\ell) \otimes E(\sigma\bar\xi_1,
\sigma\bar\xi_2) \otimes P_p(\sigma\bar\tau_k \mid k\ge2) \,.
\end{aligned}
\end{equation*}
The inclusion of $0$-simplices $\eta \: B \to \thh(B)$ is split for
commutative $B$ by the augmentation $\epsilon \: \thh(B) \to B$.  Thus
there are unique representatives in B{\"o}kstedt filtration~$1$, with
zero augmentation, for each of the classes $\sigma x$.  
There are
multiplicative extensions $(\sigma\bar\tau_k)^p = \sigma\bar\tau_{k+1}$
for $k\ge0$, see \cite{kkmodp-AnR05}*{Prop.~5.9}, so
\begin{equation*}
\begin{aligned}
H_*(\thh(\Z/p)) &= A_* \otimes P(\sigma\bar\tau_0) \\
H_*(\thh(\Z_{(p)})) &= H_*(H\Z_{(p)}) \otimes E(\sigma\bar\xi_1) \otimes
P(\sigma\bar\tau_1) \\
H_*(\thh(\ell)) &= H_*(\ell) \otimes E(\sigma\bar\xi_1, \sigma\bar\xi_2)
\otimes P(\sigma\bar\tau_2)
\end{aligned}
\end{equation*}
as $A_*$-comodule algebras.  The $A_*$-comodule coactions are given by
\begin{equation}
\begin{aligned}
\nu(\sigma\bar\tau_0) &= 1 \otimes \sigma\bar\tau_0 \\
\nu(\sigma\bar\xi_1) &= 1 \otimes \sigma\bar\xi_1 \\
\nu(\sigma\bar\tau_1) &= 1 \otimes \sigma\bar\tau_1  + \bar\tau_0
\otimes \sigma\bar\xi_1 \\
\nu(\sigma\bar\xi_2) &= 1 \otimes \sigma\bar\xi_2 \\
\nu(\sigma\bar\tau_2) &= 1 \otimes \sigma\bar\tau_2  + \bar\tau_0
\otimes \sigma\bar\xi_2 \,.
\end{aligned}
\label{kkmodp-eq:4.2}
\end{equation}
The natural map $\pi_* \: \thh(\ell) \to \thh(\Z_{(p)})$ induced by $\pi \:
\ell \to \Z_{(p)}$ takes $\sigma\bar\xi_2$ to $0$ and $\sigma\bar\tau_2$
to $(\sigma\bar\tau_1)^p$.  The natural map $i_* \: \thh(\Z_{(p)}) \to
\thh(\Z/p)$ induced by $i \: \Z_{(p)} \to \Z/p$ takes $\sigma\bar\xi_1$
to $0$ and $\sigma\bar\tau_1$ to $(\sigma\bar\tau_0)^p$.

The B{\"o}kstedt spectral sequence for the associative $S$-algebra
$B = \ell/p$ begins
\begin{equation*}
E^2(\ell/p) = H_*(\ell/p) \otimes E(\sigma\bar\xi_k \mid k\ge1)
\otimes \Gamma(\sigma\bar\tau_0, \sigma\bar\tau_k \mid k\ge2) \,.
\end{equation*}
It is an $A_*$-comodule module spectral sequence over the B{\"o}kstedt
spectral sequence for~$\ell$, since the $\ell$-algebra multiplication
$\ell \wedge \ell/p \to \ell/p$ is a map of associative $S$-algebras.
However, it is not itself an algebra spectral sequence, since the
product on $\ell/p$ is not commutative enough to induce a natural
product structure on $\thh(\ell/p)$.  Nonetheless, we will use the
algebra structure present at the $E^2$-term to help in naming classes.

The map $\pi \: \ell/p \to H\Z/p$ induces an injection of B{\"o}kstedt
spectral sequence $E^2$-terms, so there are differentials generated
algebraically by
\begin{equation*}
d^{p-1}(\gamma_j\sigma\bar\tau_k) \doteq \sigma\bar\xi_{k+1}
\cdot \gamma_{j-p}\sigma\bar\tau_k
\end{equation*}
for $j\ge p$, $k=0$ or $k\ge2$, leaving
\begin{equation}
E^\infty(\ell/p) = H_*(\ell/p) \otimes E(\sigma\bar\xi_2) \otimes
P_p(\sigma\bar\tau_0, \sigma\bar\tau_k \mid k\ge2)
\label{kkmodp-eq:4.3}
\end{equation}
as an $A_*$-comodule module over $E^\infty(\ell)$.  In order to
obtain $H_*(\thh(\ell/p))$,  we need to resolve the $A_*$-comodule
and $H_*(\thh(\ell))$-module extensions.  This is achieved in
Lemma~\ref{kkmodp-lem:4.6} below.

The natural map $\pi_* \: E^\infty(\ell/p) \to E^\infty(\Z/p)$
is an isomorphism in total degrees $\le (2p-2)$ and injective in
total degrees $\le (2p^2-2)$.  The first class in the kernel is
$\sigma\bar\xi_2$.  Hence there are unique classes
\begin{equation*}
1 \ ,\  \bar\tau_0 \ ,\  \sigma\bar\tau_0 \ ,\  \bar\tau_0
\sigma\bar\tau_0 \ ,\  \dots \ ,\ (\sigma\bar\tau_0)^{p-1}
\end{equation*}
in degrees $0 \le * \le 2p-2$ of $H_*(\thh(\ell/p))$, mapping to classes
with the same names in $H_*(\thh(\Z/p))$.  More concisely, these are the
monomials $\bar\tau_0^\delta (\sigma\bar\tau_0)^i$ for $0\le\delta\le1$
and $0\le i\le p-1$, except that the degree~$(2p-1)$ case $(\delta,i) =
(1,p-1)$ is omitted.  The $A_*$-comodule coaction on these classes is
given by the same formulas in $H_*(\thh(\ell/p))$ as in $H_*(\thh(\Z/p))$,
cf.~\eqref{kkmodp-eq:4.2}.

There is also a class $\bar\xi_1$ in degree~$(2p-2)$ of
$H_*(\thh(\ell/p))$
mapping to a class with the same name, and same $A_*$-coaction, in
$H_*(\thh(\Z/p))$.

In degree~$(2p-1)$, $\pi_*$ is a map of extensions from
\begin{equation*}
0 \to \F_p\{\bar\xi_1 \bar\tau_0\} \to H_{2p-1}(\thh(\ell/p)) \to
\F_p\{\bar\tau_0 (\sigma\bar\tau_0)^{p-1}\} \to 0
\end{equation*}
to
\begin{equation*}
0 \to \F_p\{\bar\tau_1, \bar\xi_1 \bar\tau_0\} \to H_{2p-1}(\thh(\Z/p))
\to \F_p\{\bar\tau_0 (\sigma\bar\tau_0)^{p-1}\} \to 0 \,.
\end{equation*}
The latter extension is canonically split by the augmentation
$\epsilon \: \thh(\Z/p) \to H\Z/p$, which uses the commutativity
of the $S$-algebra $H\Z/p$.

In degree~$2p$, the map $\pi_*$ goes from
\begin{equation*}
H_{2p}(\thh(\ell/p)) = \F_p\{\bar\xi_1 \sigma\bar\tau_0\}
\end{equation*}
to
\begin{equation*}
0 \to \F_p\{\bar\tau_0 \bar\tau_1\} \to H_{2p}(\thh(\Z/p)) \to
\F_p\{\sigma\bar\tau_1, \bar\xi_1 \sigma\bar\tau_0\} \to 0 \,.
\end{equation*}
Again the latter extension is canonically split.

\begin{lemma}
There is a unique class $y$ in $H_{2p-1}(\thh(\ell/p))$ represented by
$\bar\tau_0 (\sigma\bar\tau_0)^{p-1}$ in $E^\infty_{p-1,p}(\ell/p)$ and
mapped by $\pi_*$ to $\bar\tau_0 (\sigma\bar\tau_0)^{p-1} - \bar\tau_1$
in $H_*(\thh(\Z/p))$.
\end{lemma}

\begin{proof}
This follows from naturality of the suspension operator $\sigma$ and
the multiplicative relation $(\sigma\bar\tau_0)^p = \sigma\bar\tau_1$
in $H_*(\thh(\Z/p))$.  A class $y$ in $H_{2p-1}(\thh(\ell/p))$ represented
by $\bar\tau_0 (\sigma\bar\tau_0)^{p-1}$ is determined modulo $\bar\xi_1
\bar\tau_0$.  Its image in $H_{2p-1}(\thh(\Z/p))$ thus has the form $\alpha
\bar\tau_1 + \bar\tau_0 (\sigma\bar\tau_0)^{p-1}$ modulo $\bar\xi_1
\bar\tau_0$, for some $\alpha \in \F_p$.  The suspension $\sigma y$
lies in $H_{2p}(\thh(\ell/p)) = \F_p\{\bar\xi_1 \sigma\bar\tau_0\}$, so
its image in $H_{2p}(\thh(\Z/p))$ is $0$ modulo $\bar\tau_0 \bar\tau_1$
and $\bar\xi_1 \sigma\bar\tau_0$.  It is also the suspension of
$\alpha \bar\tau_1 + \bar\tau_0 (\sigma\bar\tau_0)^{p-1}$ modulo
$\bar\xi_1 \bar\tau_0$, which equals $\sigma(\alpha \bar\tau_1) +
(\sigma\bar\tau_0)^p = (\alpha+1) \sigma\bar\tau_1$.  In particular, the
coefficient $(\alpha + 1)$ of $\sigma\bar\tau_1$ is $0$, so $\alpha=-1$.
\end{proof}

Let
\begin{equation*}
H_*(\thh(\ell))/(\sigma\bar\xi_1) =
  H_*(\ell) \otimes E(\sigma\bar\xi_2) \otimes P(\sigma\bar\tau_2)
\end{equation*}
denote the quotient algebra of $H_*(\thh(\ell))$ by the ideal generated
by $\sigma\bar\xi_1$.

\begin{lemma}\label{kkmodp-lem:4.6}
The classes 
\begin{equation*}
1 \ ,\  \bar\tau_0 \ ,\  \sigma\bar\tau_0 \ ,\  \bar\tau_0
\sigma\bar\tau_0 \ ,\  \dots \ ,\ (\sigma\bar\tau_0)^{p-1} \ ,\
\bar\tau_0 (\sigma\bar\tau_0)^{p-1} \,,
\end{equation*}
in $E^\infty(\ell/p)$ represent unique homology classes in
$H_*(\thh(\ell/p))$, which by abuse of notation will be
denoted 
\begin{equation*}
1 \ ,\  \bar\tau_0 \ ,\  \sigma\bar\tau_0 \ ,\  \bar\tau_0
\sigma\bar\tau_0 \ ,\  \dots \ ,\ (\sigma\bar\tau_0)^{p-1} \
,\  y \,,
\end{equation*}
mapping under $\pi_*$ to classes with the same names in $H_*(\thh(\Z/p))$,
except for~$y$, which maps to
\begin{equation*}
\bar\tau_0 (\sigma\bar\tau_0)^{p-1} - \bar\tau_1 \,.
\end{equation*}
The graded $H_*(\thh(\ell))$-module $H_*(\thh(\ell/p))$ is
a free $H_*(\thh(\ell))/(\sigma\bar\xi_1)$-module of rank
$2p$ generated by these classes in degrees $0$
through~$2p-1$:
\begin{equation*} 
H_*(\thh(\ell/p)) =  H_*(\thh(\ell))/(\sigma\bar\xi_1)
  \otimes \F_p\{1, \bar\tau_0, \sigma\bar\tau_0, \bar\tau_0
  \sigma\bar\tau_0, \dots, (\sigma\bar\tau_0)^{p-1}, y\} \,.
\end{equation*}
The $A_*$-comodule coactions are given by
\begin{equation*}
\nu((\sigma\bar\tau_0)^i) = 1 \otimes (\sigma\bar\tau_0)^i
\end{equation*}
for $0\le i\le p-1$,
\begin{equation*}
\nu(\bar\tau_0 (\sigma\bar\tau_0)^i) = 1 \otimes \bar\tau_0
  (\sigma\bar\tau_0)^i + \bar\tau_0 \otimes (\sigma\bar\tau_0)^i
\end{equation*}
for $0 \le i \le p-2$, and
\begin{equation*}
\nu(y) = 1 \otimes y + \bar\tau_0 \otimes (\sigma\bar\tau_0)^{p-1} -
  \bar\tau_0 \otimes \bar\xi_1 - \bar\tau_1 \otimes 1 \,.
\end{equation*}
\end{lemma}

\begin{proof}
$H_*(\ell/p)$ is freely generated as a module over $H_*(\ell)$ by $1$
and $\bar\tau_0$, and the classes $\sigma\bar\xi_2$ and
$\sigma\bar\tau_2$ in $H_*(\thh(\ell))$ induce multiplication by the
same symbols in $E^\infty(\ell/p)$, as given in~\eqref{kkmodp-eq:4.3}.  This
generates all of $E^\infty(\ell/p)$ from the $2p$ classes
$\bar\tau_0^\delta (\sigma\bar\tau_0)^i$ for $0\le\delta\le1$
and $0\le i\le p-1$.

We claim that multiplication by $\sigma\bar\xi_1$ acts trivially on
$H_*(\thh(\ell/p))$.  It suffices to verify this on the module
generators $\bar\tau_0^\delta (\sigma\bar\tau_0)^i$, for which the
product with $\sigma\bar\xi_1$ remains in the range of degrees where
the map to $H_*(\thh(\Z/p))$ is injective.  The action of
$\sigma\bar\xi_1$ is trivial on $H_*(\thh(\Z/p))$, since
$d^{p-1}(\gamma_p \sigma\bar\tau_0) \doteq \sigma\bar\xi_1$ and
$\epsilon(\sigma\bar\xi_1) = 0$, and this implies the claim.

The $A_*$-comodule coaction on each module generator, including $y$, is
determined by that on its image under $\pi_*$.
In the latter case, for example, we have
\begin{equation*}
\begin{aligned}
(1 \otimes \pi_*)(\nu(y)) &= \nu(\pi_*(y)) =
	\nu(\bar\tau_0(\sigma\bar\tau_0)^{p-1}-\bar\tau_1) \\
&= 1 \otimes \bar\tau_0(\sigma\bar\tau_0)^{p-1} + \bar\tau_0 \otimes
	(\sigma\bar\tau_0)^{p-1} - 1 \otimes \bar\tau_1 - \bar\tau_0
	\otimes \bar\xi_1 - \bar\tau_1 \otimes 1 \\
&= (1 \otimes \pi_*)(1 \otimes y + \bar\tau_0 \otimes
(\sigma\bar\tau_0)^{p-1} - \bar\tau_0 \otimes \bar\xi_1 -
\bar\tau_1 \otimes 1) \,,
\end{aligned}
\end{equation*} 
and this proves our formula for $\nu(y)$ since
$1\otimes \pi_*$ is injective in this degree.
\end{proof}

\begin{remark}
  Notice that Lemma~\ref{kkmodp-lem:4.6} implies that for different
  choices of 
  $\ell$-module structure on $\ell/p$, the resulting
  homology groups $H_*(\thh(\ell/p))$ are (abstractly) isomorphic as
  graded $H_*(\thh(\ell))$-modules and $A_*$-comodules.
\end{remark}

\section{Passage to $V(1)$-homotopy }
\label{kkmodp-sec:v1}

For $p\ge5$ the Smith--Toda complex $V(1) = S \cup_p e^1
\cup_{\alpha_1} e^{2p-1} \cup_p e^{2p}$ is a homotopy commutative
ring spectrum~\cite{kkmodp-Sm70}*{Th~5.1}, \cite{kkmodp-Ok84}*{Ex.~4.5}.  It is defined as the
mapping cone of the Adams self-map $v_1 \: \Sigma^{2p-2}
V(0) \to V(0)$ of the mod~$p$ Moore spectrum $V(0) = S \cup_p e^1$.
Hence there is a cofiber sequence
\begin{equation*}
\Sigma^{2p-2} V(0) \xr{v_1} V(0) \xr{i_1} V(1) \xr{j_1}
	\Sigma^{2p-1} V(0) \,.
\end{equation*}
There are some choices of orientations involved in fixing such an exact
triangle, compare for instance with~\cite{kkmodp-HM03}*{Sect.~2.1}.
The composite map $\beta_{1,1} = i_1 j_1 \: V(1) \to
\Sigma^{2p-1} V(1)$ defines the primary $v_1$-Bockstein
homomorphism, acting naturally on $V(1)_*(X)$.

In this section we compute $V(1)_* \thh(\ell/p)$ as a
module over $V(1)_* \thh(\ell)$, for any prime $p\ge5$.  The unique ring
spectrum map from $V(1)$ to $H\Z/p$ induces the identification
\begin{equation*}
H_*(V(1)) = E(\tau_0, \tau_1)
 \end{equation*}
(no conjugations) as $A_*$-comodule subalgebras of $A_*$,
see~\cite{kkmodp-To71}*{\S 4}.  Here
\begin{equation*}
\begin{aligned}
\nu(\tau_0) &= 1 \otimes \tau_0 + \tau_0 \otimes 1 \\
\nu(\tau_1) &= 1 \otimes \tau_1 + \xi_1 \otimes \tau_0 + \tau_1 \otimes 1
\,.
\end{aligned}
\end{equation*}

A form of the following lemma goes back to \cite{kkmodp-Wh62}*{p.~271}.

\begin{lemma}\label{kkmodp-lem:4.45}
Let $M$ be any $H\Z/p$-module spectrum.  Then $M$ is equivalent to a wedge
sum of suspensions of $H\Z/p$.
Hence $H_*(M)$ is a sum of shifted copies of $A_*$ as an $A_*$-comodule,
and the Hurewicz homomorphism $\pi_*(M) \to H_*(M)$ identifies $\pi_*(M)$
with the $A_*$-comodule primitives in $H_*(M)$.
\end{lemma}

\begin{proof}
The module action map $\lambda \: H\Z/p \wedge M \to M$ is a retraction,
so $\pi_*(M)$ is a direct summand of $\pi_*(H\Z/p \wedge M) = H_*(M)$,
hence is a graded $\Z/p$-vector space.  Choose maps $\alpha \: S^n \to M$
that represent a basis for this vector space.  The wedge sum of the maps
\begin{equation*}
\lambda \circ (1\wedge\alpha) \: \Sigma^n H\Z/p = H\Z/p \wedge S^n
\to M
\end{equation*}
is the desired $\pi_*$-isomorphism $\bigvee_{\alpha} \Sigma^n
H\Z/p \to M$.
\end{proof}

For each $\ell$-algebra $B$, $V(1) \wedge \thh(B)$ is a module spectrum
over $V(1) \wedge \thh(\ell)$ and thus over $V(1) \wedge \ell \simeq
H\Z/p$, so $H_*(V(1) \wedge \thh(B))$ is a sum of copies of $A_*$ as
an $A_*$-comodule, by Lemma~\ref{kkmodp-lem:4.45}.  In particular, $V(1)_* \thh(B) = \pi_*(V(1) \wedge
\thh(B))$ is naturally identified with the subgroup of $A_*$-comodule
primitives
in
\begin{equation*}
H_*(V(1) \wedge \thh(B)) \cong H_*(V(1)) \otimes H_*(\thh(B))
\end{equation*}
with the diagonal $A_*$-comodule coaction.  We write $v \wedge x$ for
the image of $v \otimes x$ under this identification, with $v \in
H_*(V(1))$ and $x \in H_*(\thh(B))$.  Let
\begin{equation}
\begin{aligned}
\epsilon_0 &= 1 \wedge \bar\tau_0 + \tau_0 \wedge 1 \\
\epsilon_1 &= 1 \wedge \bar\tau_1 + \tau_0 \wedge \bar\xi_1 + \tau_1
\wedge 1 \\
\lambda_1 &= 1 \wedge \sigma\bar\xi_1 \\
\lambda_2 &= 1 \wedge \sigma\bar\xi_2 \\
\mu_0 &= 1 \wedge \sigma\bar\tau_0 \\
\mu_1 &= 1 \wedge \sigma\bar\tau_1 + \tau_0 \wedge \sigma\bar\xi_1 \\
\mu_2 &= 1 \wedge \sigma\bar\tau_2 + \tau_0 \wedge \sigma\bar\xi_2 \,.
\end{aligned}
\label{kkmodp-eq:5.1}
\end{equation}
These are all $A_*$-comodule primitive, when defined, in $H_*(V(1) \wedge
\thh(B))$ for $B=\ell$, $\ell/p$, $H\Z_p$ or $H\Z/p$ (see
Remark~\ref{kkmodp-rem:3.99}).  By a dimension
count,
\begin{equation*}
\begin{aligned}
V(1)_* \thh(\Z/p) &= E(\epsilon_0, \epsilon_1) \otimes P(\mu_0) \\
V(1)_* \thh(\Z_{(p)}) &= E(\epsilon_1) \otimes E(\lambda_1) \otimes P(\mu_1) \\
V(1)_* \thh(\ell) &= E(\lambda_1, \lambda_2) \otimes P(\mu_2)
\end{aligned}
\end{equation*}
as commutative $\F_p$-algebras.  The map $\pi \: \ell \to H\Z_{(p)}$
takes $\lambda_2$ to $0$ and $\mu_2$ to $\mu_1^p$.  The map $i \:
H\Z_{(p)} \to H\Z/p$ takes $\lambda_1$ to $0$ and $\mu_1$ to
$\mu_0^p$.  Note that $\mu_2 \in V(1)_{2p^2} \thh(\ell)$ was
simply denoted $\mu$ in \cite{kkmodp-AR02}.

In degrees $\le (2p-2)$ of $H_*(V(1) \wedge \thh(\ell/p))$ the classes
\begin{equation}
\mu_0^i := 1 \wedge (\sigma\bar\tau_0)^i
\label{kkmodp-eq:5.3.a}
\end{equation}
for $0\le i\le p-1$ and
\begin{equation}
\epsilon_0 \mu_0^i := 1 \wedge \bar\tau_0 (\sigma\bar\tau_0)^i + \tau_0
\wedge (\sigma\bar\tau_0)^i
\label{kkmodp-eq:5.3.b}
\end{equation}
for $0\le i\le p-2$ are $A_*$-comodule primitive, hence lift uniquely
to $V(1)_* \thh(\ell/p)$.  These map to the classes $\epsilon_0^\delta
\mu_0^i$ in $V(1)_* \thh(\Z/p)$ for $0\le\delta\le1$ and $0\le
i\le p-1$, except that the degree bound excludes the top case of
$\epsilon_0\mu_0^{p-1}$.

In degree $(2p-1)$ of $H_*(V(1) \wedge \thh(\ell/p))$ we have generators
$1 \wedge \bar\xi_1 \bar\tau_0$, $\tau_0 \wedge
(\sigma\bar\tau_0)^{p-1}$, $\tau_0 \wedge \bar\xi_1$, $\tau_1 \wedge 1$
and $1 \wedge y$.  These have coactions
\begin{equation*}
\begin{aligned}
\nu(1 \wedge \bar\xi_1 \bar\tau_0) &= 1 \otimes 1 \wedge \bar\xi_1
\bar\tau_0 + \bar\tau_0 \otimes 1 \wedge \bar\xi_1 + \bar\xi_1 \otimes
1 \wedge \bar\tau_0 + \bar\xi_1 \bar\tau_0 \otimes 1 \wedge 1 \\
\nu(\tau_0 \wedge (\sigma\bar\tau_0)^{p-1}) &= 1 \otimes \tau_0 \wedge 
(\sigma\bar\tau_0)^{p-1} + \tau_0 \otimes 1 \wedge (\sigma\bar\tau_0)^{p-1} \\
\nu(\tau_0 \wedge \bar\xi_1) &= 1 \otimes \tau_0 \wedge \bar\xi_1 +
\tau_0 \otimes 1 \wedge \bar\xi_1 + \bar\xi_1 \otimes \tau_0 \wedge 1
+ \bar\xi_1\tau_0 \otimes 1 \wedge 1 \\
\nu(\tau_1 \wedge 1) &= 1 \otimes \tau_1 \wedge 1 + \xi_1 \otimes \tau_0
\wedge 1 + \tau_1 \otimes 1 \wedge 1
\end{aligned}
\end{equation*}
and
\begin{equation*}
\nu(1 \wedge y) = 1 \otimes 1 \wedge y + \bar\tau_0 \otimes 1 \wedge
	(\sigma\bar\tau_0)^{p-1} - \bar\tau_0 \otimes 1 \wedge \bar\xi_1 -
	\bar\tau_1 \otimes 1 \wedge 1 \,.
\end{equation*}
Hence the sum
\begin{equation}
\bar\epsilon_1 := 1 \wedge y + \tau_0 \wedge
(\sigma\bar\tau_0)^{p-1} - \tau_0 \wedge \bar\xi_1 - \tau_1 \wedge 1
\label{kkmodp-eq:5.3.c}
\end{equation}
is $A_*$-comodule primitive.  Its image under $\pi_*$ in $H_*(V(1)
\wedge \thh(\Z/p))$
is
\begin{equation*}
\epsilon_0 \mu_0^{p-1} - \epsilon_1 = 1 \wedge \bar\tau_0
(\sigma\bar\tau_0)^{p-1} + \tau_0 \wedge (\sigma\bar\tau_0)^{p-1} - 1
\wedge \bar\tau_1 - \tau_0 \wedge \bar\xi_1 - \tau_1 \wedge 1 \,.
\end{equation*}

Let
\begin{equation*}
V(1)_* \thh(\ell)/(\lambda_1) = E(\lambda_2) \otimes P(\mu_2)
\end{equation*}
be the quotient algebra of $V(1)_* \thh(\ell)$ by the ideal generated
by $\lambda_1$.

\begin{proposition}\label{kkmodp-prop:5.4}
The classes 
\begin{equation*} 
1 \ ,\  \epsilon_0 \ ,\  \mu_0 \ ,\  \epsilon_0 \mu_0 \ ,\  \dots \ ,\
\mu_0^{p-1} \ ,\  \bar\epsilon_1\in H_*(V(1) \wedge
\thh(\ell/p)) 
\end{equation*}
defined in~\eqref{kkmodp-eq:5.3.a}, \eqref{kkmodp-eq:5.3.b}
and~\eqref{kkmodp-eq:5.3.c} have unique lifts with same names
in $V(1)_* \thh(\ell/p)$. The graded $V(1)_*
\thh(\ell)$-module $V(1)_* \thh(\ell/p)$ is a free
$V(1)_* \thh(\ell)/(\lambda_1)$-mod\-ule generated by these $2p$
classes: 
\begin{equation*} 
V(1)_* \thh(\ell/p) = V(1)_* \thh(\ell)/(\lambda_1) \otimes \F_p\{1,
\epsilon_0, \mu_0, \epsilon_0 \mu_0, \dots, \mu_0^{p-1}, \bar\epsilon_1\}
\,.
\end{equation*}
The map $\pi_*$ to $V(1)_* \thh(\Z/p)$
takes $\epsilon_0^\delta \mu_0^i$ in degree $0 \le \delta + 2i \le 2p-2$
to $\epsilon_0^\delta \mu_0^i$, and takes $\bar\epsilon_1$ in degree
$(2p-1)$ to $\epsilon_0 \mu_0^{p-1} - \epsilon_1$.
\end{proposition}

\begin{proof}
Additively, this follows by another dimension count, and the
description of $\pi_*$ follows from the definition of
the classes in question. It remains to prove that the
action of $V(1)_* \thh(\ell)$ is as claimed.

The action of $\mu_2^i$ and
$\lambda_2\mu_2^i$ in $V(1)_*\thh(\ell)$ on the generators 
\begin{equation*}
1, \epsilon_0, \mu_0, \epsilon_0 \mu_0, \dots,
\mu_0^{p-1}, \bar\epsilon_1
\end{equation*}
of $V(1)_*\thh(\ell/p)$
is non-trivial for all $i\ge0$, since 
the corresponding statement holds for the images of these
classes in $H_*(V(1)\wedge \thh(\ell))$ and $H_*(V(1)\wedge
\thh(\ell/p))$. This follows from Lemma~\ref{kkmodp-lem:4.6} and the
definition these classes. It remains to show that
$\lambda_1$ acts trivially on $V(1)_* \thh(\ell/p)$. For
degree reasons, multiplication by $\lambda_1$ is zero on all
classes except possibly $\mu_2^i$ and $\lambda_2\mu_2^i$,
for $i\ge0$.
Because of the module structure, it suffices to shows that
$\lambda_1=\lambda_1\cdot1=0$ in  $V(1)_* \thh(\ell/p)$. 
This follows from the statement that
the image of $\lambda_1$ in $H_*(V(1)\wedge \thh(\ell/p))$ 
is equal to $1\wedge \sigma\bar\xi_1=0$, as implied by
Lemma~\ref{kkmodp-lem:4.6}. 
\end{proof}

\section{The $C_p$-Tate construction }
\label{kkmodp-sec:tate}
For the remainder of this paper, let $p$ be a prime with
$p\ge5$.
We briefly recall the terminology on equivariant stable homotopy theory
used in the sequel, and refer
to~\cite{kkmodp-GM95}, \cite{kkmodp-HM97}*{\S1}, \cite{kkmodp-HM03}*{\S4}
and~\cite{kkmodp-AR02}*{\S3} for more details.
Let $C_{p^n}$ denote the cyclic group of order $p^n$, considered
as a closed subgroup of the circle group $S^1$, and let
$G=S^1$ or $C_{p^n}$.  For each spectrum $X$
with $S^1$-action, let $X_{hG} = EG_+ \wedge_G X$ and $X^{hG} =
F(EG_+, X)^G$
denote its homotopy orbit and homotopy fixed point spectra, as usual.
We now write $X^{tG} = [\widetilde{EG} \wedge F(EG_+, X)]^G$ for the
$G$-Tate construction on~$X$, which was denoted $t_G(X)^G$ in \cite{kkmodp-GM95}
and $\hat\H(G, X)$ in \cites{kkmodp-HM97,kkmodp-HM03,kkmodp-AR02}. 

We denote by $F$ the Frobenius map 
$X^{C_{p^n}}\to X^{C_{p^{n-1}}}$ given by the
inclusion of fixed-point spectra, and by $V$ the
Verschiebung map $X^{C_{p^{n-1}}}\to X^{C_{p^n}}$ given by
transfer. We shall also consider the homotopy Frobenius, Tate
Frobenius and homotopy
Verschiebung maps 
$F^h \: X^{hS^1}\to X^{hC_{p^n}}$, $F^h \: X^{hC_{p^n}}\to
X^{hC_{p^{n-1}}}$, $F^t \: X^{tS^1}\to X^{tC_{p^n}}$  and $V^h \: X^{hC_{p^{n-1}}}\to
X^{hC_{p^n}}$. 

There are conditionally
convergent $G$-homotopy fixed point
and $G$-Tate spectral sequences in $V(1)$-homotopy for $X$, with
\begin{equation*}
E^2_{s,t}(G, X) = H_{gp}^{-s}(G; V(1)_t(X))
\Longrightarrow V(1)_{s+t}(X^{hG})
\end{equation*}
and
\begin {equation*}
\hat E^2_{s,t}(G, X) = \hat H_{gp}^{-s}(G; V(1)_t(X))
\Longrightarrow V(1)_{s+t}(X^{tG}) \,.
\end{equation*}
Here $H_{gp}^{*}(G; V(1)_*(X))$ denotes the group cohomology
of $G$ and $\hat H_{gp}^{*}(G; V(1)_*(X))$ the Tate
cohomology of $G$, with coefficients in $V(1)_*(X)$. 
Notice that in our case, with $X=\thh(B)$,
the action of $G$ on $V(1)_*(X)$ is trivial,
since it is the restriction of an $S^1$-action.
We write $H_{gp}^*(C_{p^n}; \F_p) = E(u_n) \otimes P(t)$ and $\hat
H_{gp}^*(C_{p^n}; \F_p) = E(u_n) \otimes P(t^{\pm1})$ with $u_n$ in degree
$1$ and $t$ in degree $2$, see for
example~\cite{kkmodp-Ben98}*{Prop.\,3.5.5}
and~\cite{kkmodp-HM03}*{Lem.~4.2.1}.  So $u_n$, $t$ and $x \in V(1)_t(X)$ have
bidegree $(-1,0)$, $(-2,0)$ and $(0,t)$ in either spectral sequence,
respectively. 
See \cite{kkmodp-HM03}*{\S4.3} for proofs of the multiplicative
properties of these spectral sequences.
Similarly, we write  $H_{gp}^*(S^1; \F_p) = P(t)$ and $\hat
H_{gp}^*(S^1; \F_p) = P(t^{\pm1})$. 
We have morphisms of spectral sequences induced by the
homotopy and Tate Frobenii, which on the $E^2$-terms map $t$ to $t$ and
$u_{n}$ to zero. 

We are principally interested in the case when $X = \thh(B)$, with the
$S^1$-action given by the cyclic
structure~\cite{kkmodp-Lo98}*{Def.~7.1.9}, \cite{kkmodp-HM03}*{\S1.2}.  It is a cyclotomic
spectrum, in the sense of \cite{kkmodp-HM97}*{\S1}, leading to the
commutative diagram
\begin{equation*}
\xymatrix{
\thh(B)_{hC_{p^n}} \ar[r]^N \ar@{=}[d] &
\thh(B)^{C_{p^n}} \ar[r]^R \ar[d]^{\Gamma_n} &
\thh(B)^{C_{p^{n-1}}} \ar[d]^{\hat\Gamma_n} \ar[r]&
\Sigma \thh(B)_{hC_{p^n}} \ar@{=}[d] \\
\thh(B)_{hC_{p^n}} \ar[r]^{N^h} &
\thh(B)^{hC_{p^n}} \ar[r]^{R^h} &
\thh(B)^{tC_{p^n}} \ar[r] &
\Sigma \thh(B)_{hC_{p^n}} 
}
\end{equation*}
of horizontal cofiber sequences.  We abbreviate $\hat E^2(G, \thh(B))$
to $\hat E^2(G, B)$, etc.  When $B$ is a commutative $S$-algebra,
this is a commutative algebra spectral sequence, and when~$B$ is an
associative $A$-algebra, with $A$ commutative, then $\hat E^*(G, B)$
is a module spectral sequence over $\hat E^*(G, A)$.  The map $R^h$
corresponds to the inclusion $E^2(G, B) \to \hat E^2(G, B)$
from the second quadrant to the upper half-plane, for connective $B$.

\begin{definition}
  We call a homomorphism of graded groups $k$-coconnected if
  it is an isomorphism in all dimensions greater than $k$
  and injective in dimension $k$.
\end{definition}

In this section we compute $V(1)_* \thh(\ell/p)^{tC_p}$ by means of the
$C_p$-Tate spectral sequence in $V(1)$-homotopy for $\thh(\ell/p)$.  In
Propositions~\ref{kkmodp-prop:6.8} and~\ref{kkmodp-prop:6.9} we show that the comparison 
map $\hat\Gamma_1 \: V(1)_* \thh(\ell/p) \to V(1)_* \thh(\ell/p)^{tC_p}$ is
$(2p-2)$-coconnected and can be identified with the algebraic
localization homomorphism that inverts $\mu_2$.

First we recall the structure of the $C_p$-Tate spectral sequence for
$\thh(\Z/p)$, with $V(0)$- and $V(1)$-coefficients.  We have $V(0)_*
\thh(\Z/p) = E(\epsilon_0) \otimes P(\mu_0)$, and (with an obvious
notation for the case of $V(0)$-homotopy) the $E^2$-terms are
\begin{equation*}
\begin{aligned}
\hat E^2(C_p, \Z/p; V(0)) &= E(u_1) \otimes P(t^{\pm1}) \otimes
E(\epsilon_0) \otimes P(\mu_0) \\
\hat E^2(C_p, \Z/p) &=  E(u_1) \otimes P(t^{\pm1}) \otimes
E(\epsilon_0, \epsilon_1) \otimes P(\mu_0) \,.
\end{aligned}
\end{equation*}
In each $G$-Tate spectral sequence we have a first differential
\begin{equation*}
d^2(x) = t \cdot \sigma x \,,
\end{equation*}
see e.g.~\cite{kkmodp-Rog98}*{\S3.3}.  We easily deduce $\sigma\epsilon_0 = \mu_0$
and $\sigma\epsilon_1 = \mu_0^p$ from~\eqref{kkmodp-eq:5.1}, so
\begin{equation*}
\begin{aligned}
\hat E^3(C_p, \Z/p; V(0)) &= E(u_1) \otimes P(t^{\pm1}) \\
\hat E^3(C_p, \Z/p) &=  E(u_1) \otimes P(t^{\pm1}) \otimes
E(\epsilon_0 \mu_0^{p-1} - \epsilon_1) \,.
\end{aligned}
\end{equation*}
Thus the $V(0)$-homotopy spectral sequence collapses at $\hat E^3 =
\hat E^\infty$.  By naturality with respect to the map $i_1 \: V(0) \to
V(1)$, all the classes on the horizontal axis of $\hat E^3(C_p, \Z/p)$
are infinite cycles, so also the latter spectral sequence collapses at
$\hat E^3(C_p, \Z/p)$.

We know from \cite{kkmodp-HM03}*{Cor.~4.4.2} that the comparison map 
\begin{equation*}
\hat\Gamma_1 \:
V(0)_* \thh(\Z/p) \to V(0)_* \thh(\Z/p)^{tC_p}
\end{equation*} 
takes $\epsilon_0^\delta
\mu_0^i$ to $(u_1 t^{-1})^\delta t^{-i}$, for all $0\le\delta\le1$,
$i\ge0$.  In particular, the integral map $\hat\Gamma_1 \: \pi_*
\thh(\Z/p) \to \pi_* \thh(\Z/p)^{tC_p}$ is
$(-2)$-cocon\-nec\-ted.
From this we can deduce the following behavior of the
comparison map $\hat\Gamma_1$ in $V(1)$-homotopy.

\begin{lemma}\label{kkmodp-lem:6.1}
The map
\begin{equation*}
\hat\Gamma_1 \: V(1)_* \thh(\Z/p) \to V(1)_* \thh(\Z/p)^{tC_p}
\end{equation*}
takes the classes $\epsilon_0^\delta \mu_0^i$ from $V(0)_* \thh(\Z/p)$,
for $0\le\delta\le1$ and $i\ge0$, to classes represented in $\hat
E^\infty(C_p, \Z/p)$ by $(u_1 t^{-1})^\delta t^{-i}$ (on the
horizontal axis).
Furthermore, it takes the class $\epsilon_0 \mu_0^{p-1} - \epsilon_1$
in degree $(2p-1)$ to a class represented by $\epsilon_0 \mu_0^{p-1} -
\epsilon_1$ (on the vertical axis).
\end{lemma}

\begin{proof}
The classes $\epsilon_0^\delta \mu_0^i$ are in the image from
$V(0)$-homotopy, and we recalled above that they are detected by
$(u_1 t^{-1})^\delta t^{-i}$ in the $V(0)$-homotopy $C_p$-Tate spectral
sequence for $\thh(\Z/p)$.  By naturality along $i_1 \: V(0) \to V(1)$,
they are detected by the same (nonzero) classes in the $V(1)$-homotopy
spectral sequence $\hat E^\infty(C_p, \Z/p)$.

To find the representative for $\hat\Gamma_1(\epsilon_0 \mu_0^{p-1} -
\epsilon_1)$ in degree~$(2p-1)$, we appeal to the cyclotomic trace map
from algebraic $K$-theory, or more precisely, to the commutative
diagram
\begin{equation}
\xymatrix{
& K(B) \ar[dr]^{tr} \ar[d]^{tr_1} \ar[dl]_{tr} \\
\thh(B) & \thh(B)^{C_p} \ar[r]^R \ar[d]^{\Gamma_1} \ar[l]_F & \thh(B)
  \ar[d]^{\hat\Gamma_1} \\
& \thh(B)^{hC_p} \ar[r]^{R^h} \ar[ul] & \thh(B)^{tC_p} \,.
}
\label{kkmodp-eq:6.2}
\end{equation}
The B{\"o}kstedt trace map $tr \: K(B) \to \thh(B)$ admits a preferred
lift $tr_n$ through each fixed point spectrum $\thh(B)^{C_{p^n}}$,
which homotopy equalizes the iterated restriction and Frobenius maps
$R^n$ and $F^n$ to $\thh(B)$, see \cite{kkmodp-Du04}*{\S3}.
In particular, the $\sigma$-operator on $V(1)_*\thh(B)$ is zero on    
classes in the image of $tr$.

In the case $B = H\Z/p$ we know that $K(\Z/p)_p \simeq H\Z_p$, so
$V(1)_* K(\Z/p) = E(\bar\epsilon_1)$, where the $v_1$-Bockstein
of $\bar\epsilon_1$ is $-1$.  The B{\"o}kstedt trace image
$tr(\bar\epsilon_1) \in V(1)_* \thh(\Z/p)$ lies in $\F_p\{\epsilon_1,
\epsilon_0 \mu_0^{p-1}\}$, has $v_1$-Bockstein $tr(-1) = -1$ and suspends
by $\sigma$ to $0$.  Hence
\begin{equation*}
tr(\bar\epsilon_1) = \epsilon_0 \mu_0^{p-1} - \epsilon_1 \,.
\end{equation*}
As we recalled above, the map $\hat\Gamma_1 \: \pi_* \thh(\Z/p) \to
\pi_* \thh(\Z/p)^{tC_p}$ is $(-2)$-cocon\-nec\-ted, so the corresponding map
in $V(1)$-homotopy is at least $(2p-2)$-coconnected.  Thus it takes
$\epsilon_0 \mu_0^{p-1} - \epsilon_1$ to a nonzero class in $V(1)_*
\thh(\Z/p)^{tC_p}$, represented somewhere in total degree~$(2p-1)$ of $\hat
E^\infty(C_p, \Z/p)$, in the lower right hand corner of the diagram.

Going down the middle part of the diagram, we reach a class $(\Gamma_1
\circ tr_1)(\bar\epsilon_1)$, represented in total degree~$(2p-1)$
in the left half-plane $C_p$-homotopy fixed point spectral sequence
$E^\infty(C_p, \Z/p)$.  Its image under the edge homomorphism
to $V(1)_* \thh(\Z/p)$ equals $(F \circ tr_1)(\bar\epsilon_1) =
tr(\bar\epsilon_1)$, hence $(\Gamma_1 \circ tr_1)(\bar\epsilon_1)$
is represented by $\epsilon_0 \mu_0^{p-1} - \epsilon_1$ in
$E^\infty_{0,2p-1}(C_p, \Z/p)$.  Its image under $R^h$ in the $C_p$-Tate
spectral sequence is the generator of $\hat E^\infty_{0,2p-1}(C_p, \Z/p)
= \F_p\{\epsilon_0 \mu_0^{p-1} - \epsilon_1\}$, hence that generator is
the $E^\infty$-representative of $\hat\Gamma_1(\epsilon_0 \mu_0^{p-1}
- \epsilon_1)$.
\end{proof}

The $(2p-2)$-connected map $\pi \: \ell/p \to H\Z/p$ induces
a $(2p-1)$-connected map $V(1)_* K(\ell/p) \to V(1)_* K(\Z/p) =
E(\bar\epsilon_1)$, by \cite{kkmodp-BM94}*{Prop.~10.9}.  
We can lift the algebraic $K$-theory class $\bar\epsilon_1$ to $\ell/p$.
This lift is not unique, but we fix one choice.

\begin{definition}
We call
\begin{equation*}
\bar\epsilon_1^K \in V(1)_{2p-1} K(\ell/p)
\end{equation*}
a chosen class that maps to the generator $\bar\epsilon_1$ in $V(1)_{2p-1} K(\Z/p)
\cong \Z/p$.
\end{definition}

\begin{lemma}\label{kkmodp-lem:6.4}
The B{\"o}kstedt trace $tr \: V(1)_* K(\ell/p) \to V(1)_* \thh(\ell/p)$
takes $\bar\epsilon_1^K$ to $\bar\epsilon_1$.
\end{lemma}

\begin{proof}
In the commutative square
\begin{equation*}
\xymatrix{
V(1)_* K(\ell/p) \ar[d]^{\pi_*} \ar[r]^-{tr} & V(1)_* \thh(\ell/p)
  \ar[d]^{\pi_*} \\
V(1)_* K(\Z/p) \ar[r]^-{tr} & V(1)_* \thh(\Z/p)
}
\end{equation*}
the trace image $tr(\bar\epsilon_1^K)$ in $V(1)_* \thh(\ell/p)$ must map
under $\pi_*$ to $tr(\bar\epsilon_1) = \epsilon_0 \mu_0^{p-1} - \epsilon_1$
in $V(1)_* \thh(\Z/p)$, which by Proposition~\ref{kkmodp-prop:5.4} characterizes it as
being equal to the class~$\bar\epsilon_1$.
Hence $tr(\bar\epsilon_1^K) = \bar\epsilon_1$.
\end{proof}

Next we turn to the $C_p$-Tate spectral sequence $\hat E^*(C_p,
\ell/p)$ in $V(1)$-homotopy for $\thh(\ell/p)$.  
Its $E^2$-term is
\begin{equation*}
\hat E^2(C_p, \ell/p) = E(u_1) \otimes P(t^{\pm1}) \otimes
\F_p\{1, \epsilon_0, \mu_0, \epsilon_0 \mu_0, \dots, \mu_0^{p-1},
  \bar\epsilon_1\} \otimes E(\lambda_2) \otimes P(\mu_2) \,.
\end{equation*}
We have $d^2(x) = t \cdot \sigma x$, where
\begin{equation*}
\sigma(\epsilon_0^\delta \mu_0^{i-1}) = 
\begin{cases}
\mu_0^i & \text{for $\delta=1$, $0<i<p$,} \\
0 & \text{otherwise}
\end{cases}
\end{equation*}
is readily deduced from~\eqref{kkmodp-eq:5.1}, and $\sigma(\bar\epsilon_1) = 0$ since
$\bar\epsilon_1$ is in the image of $tr$.  Thus
\begin{equation}
\hat E^3(C_p, \ell/p) = E(u_1) \otimes P(t^{\pm1}) \otimes
E(\bar\epsilon_1) \otimes E(\lambda_2) \otimes P(t\mu_2) \,.
\label{kkmodp-eq:6.5}
\end{equation}
We prefer to use $t\mu_2$ rather than $\mu_2$ as a generator, since it
represents multiplication by $v_2$ (up to a unit factor in
$\F_p$) in all module spectral sequences
over $E^*(S^1, \ell)$, by \cite{kkmodp-AR02}*{Prop.~4.8}.

To proceed, we shall use that $\hat E^*(C_p ,\ell/p)$ is a module
over the spectral sequence for $\thh(\ell)$.  We therefore recall the
structure of the latter spectral sequence, from
\cite{kkmodp-AR02}*{Th.~5.5}.
It begins
\begin{equation*}
\hat E^2(C_p, \ell) =
E(u_1) \otimes P(t^{\pm1}) \otimes E(\lambda_1, \lambda_2)
\otimes P(\mu_2) \,.
\end{equation*}
The classes $\lambda_1$, $\lambda_2$ and $t\mu_2$ are infinite
cycles, and the differentials
\begin{equation*}
\begin{aligned}
d^{2p}(t^{1-p}) &\doteq t\lambda_1 \\
d^{2p^2}(t^{p-p^2}) &\doteq t^p\lambda_2 \\
d^{2p^2+1}(u_1t^{-p^2}) &\doteq t\mu_2
\end{aligned}
\end{equation*}
leave the terms
\begin{equation*}
\begin{aligned}
\hat E^{2p+1}(C_p, \ell) &= E(u_1, \lambda_1, \lambda_2)
	\otimes P(t^{\pm p}, t\mu_2) \\
\hat E^{2p^2+1}(C_p, \ell) &= E(u_1, \lambda_1, \lambda_2)
	\otimes P(t^{\pm p^2}, t\mu_2) \\
\hat E^{2p^2+2}(C_p, \ell) &= E(\lambda_1, \lambda_2)
	\otimes P(t^{\pm p^2})
\end{aligned}
\end{equation*}
with $\hat E^{2p^2+2} = \hat E^\infty$, converging to $V(1)_*
\thh(\ell)^{tC_p}$.  The comparison map $\hat\Gamma_1$ takes $\lambda_1$,
$\lambda_2$ and $\mu_2$ to $\lambda_1$, $\lambda_2$ and
$t^{-p^2}$ (up to a unit factor in $\F_p$),
respectively, inducing the algebraic localization map and identification
\begin{equation*}
\hat\Gamma_1 \: V(1)_* \thh(\ell) \to
V(1)_* \thh(\ell) [\mu_2^{-1}] \cong V(1)_* \thh(\ell)^{tC_p} \,.
\end{equation*}

\begin{lemma}\label{kkmodp-lem:6.6}
In $\hat E^*(C_p, \ell/p)$, the class
$u_1 t^{-p}$ supports the nonzero differential
\begin{equation*}
d^{2p^2}(u_1 t^{-p}) \doteq u_1 t^{p^2-p} \lambda_2 \,,
\end{equation*}
and does not survive to the $E^\infty$-term.
\end{lemma}

\begin{proof}
In $\hat E^*(C_p, \ell)$, there is such a differential.
By naturality along $i \: \ell \to \ell/p$, it follows that there is
also such a differential in $\hat E^*(C_p, \ell/p)$.  It remains to
argue that the target class is nonzero at the $E^{2p^2}$-term.
Considering the $E^3$-term in~\eqref{kkmodp-eq:6.5},
the only possible source of a previous differential hitting $u_1
t^{p^2-p} \lambda_2$ is $\bar\epsilon_1$, supporting a
$d^{2p^2-2p+1}$-differential.  But $\bar\epsilon_1$ is in
an even column and $u_1 t^{p^2-p} \lambda_2$ is in an odd column.  By
naturality with respect to the Tate Frobenius map
$F^t \: \thh(\ell/p)^{tS^1}\to \thh(\ell/p)^{tC_p}$, 
any such differential from an even to an odd column must be zero.
Indeed, the $S^1$-Tate spectral sequence
has $E^2$-term given by $P(t^{\pm1})\otimes
V(1)_*\thh(\ell/p)$, and $F^t$ induces the
injective homomorphism that takes $\hat E^2(S^1, \ell/p)$ isomorphically
to the even columns
of $\hat E^2(C_p, \ell/p)$. Since $\hat E^*(S^1, \ell/p)$ is
concentrated in even columns, all differentials of odd
length are zero. By naturality,
classes of $\hat E^r(C_p, \ell/p)$ that lie in the
image of  $\hat E^r(F^t)$ cannot support a differential of
odd length; compare with~\cite{kkmodp-AR02}*{Lemma~5.2}.
In the present situation, the $d^2$-differential of $\hat E^*(C_p, \ell/p)$
leading to~\eqref{kkmodp-eq:6.5} is also non-zero in  
$\hat E^*(S^1, \ell/p)$, so that we have
\begin{equation*}
\hat  E^3(S^1, \ell/p)=
P(t^{\pm1}) \otimes E(\bar\epsilon_1) \otimes E(\lambda_2)
\otimes P(t\mu_2)\,.
\end{equation*}
By inspection, if the class $\bar\epsilon_1\in\hat E^2(C_p,
\ell/p)$ survives to $\hat E^{2p^2-2p+1}(C_p,\ell/p)$, then
it will lie in the image of $\hat E^{2p^2-2p+1}(F^t)$. 
\end{proof}

To determine the map $\hat\Gamma_1$ we use naturality with
respect to the map $\pi \: \ell/p \to H\Z/p$. 

\begin{lemma}\label{kkmodp-lem:6.7}
The classes $1, \epsilon_0, \mu_0, \epsilon_0 \mu_0, \dots, \mu_0^{p-1}$
and $\bar\epsilon_1$ in $V(1)_* \thh(\ell/p)$ map under~$\hat\Gamma_1$
to classes in $V(1)_* \thh(\ell/p)^{tC_p}$ that are represented in $\hat
E^\infty(C_p, \ell/p)$ by the permanent cycles $(u_1 t^{-1})^\delta
t^{-i}$ (on the horizontal axis) in degrees $\le (2p-2)$, and by the
permanent cycle $\bar\epsilon_1$ (on the vertical axis) in degree
$(2p-1)$.
\end{lemma}

\begin{proof}
In the commutative square
\begin{equation*}
\xymatrix{
V(1)_* \thh(\ell/p) \ar[r]^-{\hat\Gamma_1} \ar[d]^{\pi_*} & V(1)_*
\thh(\ell/p)^{tC_p} \ar[d]^{\pi_*} \\
V(1)_* \thh(\Z/p) \ar[r]^-{\hat\Gamma_1} & V(1)_* \thh(\Z/p)^{tC_p}
}
\end{equation*}
the classes $\epsilon_0^\delta \mu_0^i$ in the upper left hand corner
map to classes in the lower right hand corner that are represented by
$(u_1 t^{-1})^\delta t^{-i}$ in degrees $\le (2p-2)$, and
$\bar\epsilon_1$ maps to $\epsilon_0 \mu_0^{p-1} - \epsilon_1$ in
degree $(2p-1)$.  This follows by combining
Proposition~\ref{kkmodp-prop:5.4} and
Lemma~\ref{kkmodp-lem:6.1}.

The first $(2p-1)$ of these are represented in maximal filtration (on
the horizontal axis), so their images in the upper right hand corner
must be represented by permanent cycles $(u_1 t^{-1})^\delta t^{-i}$ in
the Tate spectral sequence $\hat E^\infty(C_p, \ell/p)$.

The image of the last class, $\bar\epsilon_1$, in the upper right hand
corner could either be represented by $\bar\epsilon_1$ in
bidegree~$(0,2p-1)$ or by $u_1 t^{-p}$ in bidegree~$(2p-1,0)$.
However, the last class supports a differential
$d^{2p^2}(u_1 t^{-p}) \doteq 
u_1 t^{p^2-p} \lambda_2$, by Lemma~\ref{kkmodp-lem:6.6} above.  This only leaves the
other possibility, that $\hat\Gamma_1(\bar\epsilon_1)$ is represented
by $\bar\epsilon_1$ in $\hat E^\infty(C_p, \ell/p)$.
\end{proof}

We proceed to determine the differential structure in $\hat E^*(C_p,
\ell/p)$, making use of the permanent cycles identified above.

\begin{proposition}\label{kkmodp-prop:6.8}
The $C_p$-Tate spectral sequence in $V(1)$-homotopy for $\thh(\ell/p)$
has
\begin{equation*}
\hat E^3(C_p, \ell/p) = E(u_1, \bar\epsilon_1, \lambda_2) \otimes
	P(t^{\pm1}, t\mu_2) \,.
\end{equation*}
It has differentials generated by
\begin{equation*}
d^{2p^2-2p+2}(t^{p-p^2} \cdot t^{-i}\bar\epsilon_1) \doteq t\mu_2
\cdot t^{-i} 
\end{equation*}
for $0<i<p$, $d^{2p^2}(t^{p-p^2}) \doteq t^p\lambda_2$ and
$d^{2p^2+1}(u_1 t^{-p^2}) \doteq t\mu_2$.  The subsequent terms are
\begin{equation*}
\begin{aligned}
\hat E^{2p^2-2p+3}(C_p, \ell/p) &= E(u_1, \lambda_2) \otimes
	\F_p\{t^{-i} \mid 0<i<p\} \otimes P(t^{\pm p}) \\
	&\qquad \oplus E(u_1, \bar\epsilon_1, \lambda_2) \otimes
	P(t^{\pm p}, t\mu_2) \\
\hat E^{2p^2+1}(C_p, \ell/p) &= E(u_1, \lambda_2) \otimes
	\F_p\{t^{-i} \mid 0<i<p\} \otimes P(t^{\pm p^2}) \\
	&\qquad \oplus E(u_1, \bar\epsilon_1, \lambda_2) \otimes
	P(t^{\pm p^2}, t\mu_2) \\
\hat E^{2p^2+2}(C_p, \ell/p) &= E(u_1, \lambda_2) \otimes
	\F_p\{t^{-i} \mid 0<i<p\} \otimes P(t^{\pm p^2}) \\
	&\qquad \oplus E(\bar\epsilon_1, \lambda_2) \otimes P(t^{\pm
	p^2}) \,.
\end{aligned}
\end{equation*}
The last term can be rewritten as
\begin{equation*}
\hat E^\infty(C_p, \ell/p) = \bigl( E(u_1) \otimes \F_p\{t^{-i} \mid
0<i<p\} \oplus E(\bar\epsilon_1) \bigr) \otimes E(\lambda_2) \otimes
P(t^{\pm p^2}) \,.
\end{equation*}
\end{proposition}

\begin{proof}
We have already identified the $E^2$- and $E^3$-terms above.  The
$E^3$-term~\eqref{kkmodp-eq:6.5} is generated over $\hat E^3(C_p, \ell)$ by an
$\F_p$-basis for $E(\bar\epsilon_1)$, so the next possible
differential is induced by $d^{2p}(t^{1-p}) \doteq t\lambda_1$.  But
multiplication by $\lambda_1$ is trivial in $V(1)_* \thh(\ell/p)$, by
Proposition~\ref{kkmodp-prop:5.4}, so $\hat E^3(C_p, \ell/p) = \hat E^{2p+1}(C_p,
\ell/p)$.
This term is generated over $\hat E^{2p+1}(C_p, \ell)$ by $P_p(t^{-1})
\otimes E(\bar\epsilon_1)$.  Here $1, t^{-1}, \dots, t^{1-p}$ and
$\bar\epsilon_1$ are permanent cycles, by
Lemma~\ref{kkmodp-lem:6.7}.  Any
$d^r$-differential before $d^{2p^2}$ must therefore originate on a
class $t^{-i} \bar\epsilon_1$ for $0<i<p$, and be of even length~$r$,
since these classes lie in even columns.  For bidegree reasons, the
first possibility is $r = 2p^2-2p+2$, so $\hat E^3(C_p, \ell/p) = \hat
E^{2p^2-2p+2}(C_p, \ell/p)$.

Multiplication by $v_2$ acts trivially on $V(1)_* \thh(\ell)$ and
$V(1)_* \thh(\ell)^{tC_p}$ for degree reasons, and therefore also on
$V(1)_* \thh(\ell/p)$ and $V(1)_* \thh(\ell/p)^{tC_p}$ by the module
structure.  The class $v_2$ maps to $t\mu_2$ in the $S^1$-Tate spectral
sequence for $\ell$, as recalled above, so multiplication by $v_2$ is
represented by multiplication by $t\mu_2$ in the $C_p$-Tate spectral
sequence for $\ell/p$.  Applied to the permanent cycles $(u_1 t^{-1})^\delta
t^{-i}$ in degrees $\le (2p-2)$, this implies that the products
\begin{equation*}
t\mu_2 \cdot (u_1 t^{-1})^\delta t^{-i}
\end{equation*}
must be infinite cycles representing zero, i.e., they must be hit by
differentials.  In the cases $\delta=1$, $0\le i\le p-2$, these classes
in odd columns cannot be hit by differentials of odd length,
such as $d^{2p^2+1}$, so the only possibility is
\begin{equation*}
d^{2p^2-2p+2}(t^{p-p^2} \cdot (u_1 t^{-1})t^{-i} \bar\epsilon_1 )
	\doteq t\mu_2 \cdot (u_1 t^{-1})t^{-i}
\end{equation*}
for $0\le i\le p-2$.  By the module structure (consider multiplication
by $u_1$) it follows that
\begin{equation*}
d^{2p^2-2p+2}(t^{p-p^2} \cdot t^{-i} \bar\epsilon_1 )
	\doteq t\mu_2 \cdot t^{-i}
\end{equation*}
for $0<i<p$.  Hence we can compute from~\eqref{kkmodp-eq:6.5} that
\begin{equation*}
\begin{aligned}
\hat E^{2p^2-2p+3}(C_p, \ell/p) &= E(u_1) \otimes P(t^{\pm p})
	\otimes \F_p\{t^{-i} \mid 0<i<p\} \otimes E(\lambda_2) \\
&\qquad \oplus E(u_1) \otimes P(t^{\pm p}) \otimes E(\bar\epsilon_1)
	\otimes E(\lambda_2) \otimes P(t\mu_2) \,.
\end{aligned}
\end{equation*}
This is generated over $\hat E^{2p+1}(C_p, \ell)$ by the permanent
cycles $1, t^{-1}, \dots, t^{1-p}$ and~$\bar\epsilon_1$, so the next
differential is induced by $d^{2p^2}(t^{p-p^2}) \doteq t^p\lambda_2$.
This leaves
\begin{equation*}
\begin{aligned}
\hat E^{2p^2+1}(C_p, \ell/p) &= E(u_1) \otimes P(t^{\pm p^2})
	\otimes \F_p\{t^{-i} \mid 0<i<p\} \otimes E(\lambda_2) \\
&\qquad \oplus E(u_1) \otimes P(t^{\pm p^2}) \otimes E(\bar\epsilon_1)
	\otimes E(\lambda_2) \otimes P(t\mu_2) \,.
\end{aligned}
\end{equation*}
Finally, $d^{2p^2+1}(u_1 t^{-p^2}) \doteq t\mu_2$ applies, and leaves
\begin{equation*}
\begin{aligned}
\hat E^{2p^2+2}(C_p, \ell/p) &= E(u_1) \otimes P(t^{\pm p^2})
	\otimes \F_p\{t^{-i} \mid 0<i<p\} \otimes E(\lambda_2) \\
&\qquad \oplus P(t^{\pm p^2}) \otimes E(\bar\epsilon_1) \otimes
	E(\lambda_2) \,.
\end{aligned}
\end{equation*}
For bidegree reasons, $\hat E^{2p^2+2} = \hat E^\infty$.
\end{proof}

\begin{proposition}\label{kkmodp-prop:6.9}
The comparison map $\hat\Gamma_1$ takes the classes 
\begin{equation*}
\textup{\it{  
$\epsilon_0^\delta
\mu_0^i$, $\bar\epsilon_1$, $\lambda_2$ and $\mu_2$ in $V(1)_*
\thh(\ell/p)$}}
\end{equation*}
to classes in $V(1)_* \thh(\ell/p)^{tC_p}$ represented by
\begin{equation*}
\textup{\it{  
$(u_1 t^{-1})^\delta t^{-i}$, $\bar\epsilon_1$, $\lambda_2$ and
$t^{-p^2}$ in $\hat E^\infty(C_p, \ell/p)$,}}
\end{equation*}
up to a unit
factor in $\F_p$, respectively.
Thus
\begin{equation*}
V(1)_* \thh(\ell/p)^{tC_p} \cong \F_p\{1, \epsilon_0, \mu_0, \epsilon_0
  \mu_0, \dots, \mu_0^{p-1}, \bar\epsilon_1\} \otimes E(\lambda_2) \otimes
  P(\mu_2^{\pm1})
 \end{equation*}
and $\hat\Gamma_1$ induces an identification 
$V(1)_* \thh(\ell/p) [\mu_2^{-1}] \cong V(1)_*
\thh(\ell/p)^{tC_p}$. In particular, $\hat\Gamma_1$ 
factors as the algebraic localization map and identification
\begin{equation*} 
\hat\Gamma_1 \: V(1)_* \thh(\ell/p) \to V(1)_* \thh(\ell/p) [\mu_2^{-1}]
\cong V(1)_* \thh(\ell/p)^{tC_p} \,,
\end{equation*}
and is $(2p-2)$-coconnected.
\end{proposition}

\begin{proof}
The image under $\hat\Gamma_1$ of the classes $1, \epsilon_0,
\mu_0, \epsilon_0 \mu_0, \dots, \mu_0^{p-1}$ and $\bar\epsilon_1$
was given in Lemma~\ref{kkmodp-lem:6.7}, and the action on the classes $\lambda_2$
and $\mu_2$ is given in the proof of \cite{kkmodp-AR02}*{Th.~5.5}.  The structure of
$V(1)_* \thh(\ell/p)^{tC_p}$ is then immediate from the $E^\infty$-term
in Proposition~\ref{kkmodp-prop:6.8}.  The top class not in the image of $\hat\Gamma_1$
is $\bar\epsilon_1 \lambda_2 \mu_2^{-1}$, in degree~$(2p-2)$.
\end{proof}

Recall that
\begin{equation*}
\begin{aligned}
TF(B;p) &= \holim_{n,F} \thh(B)^{C_{p^n}} \\
TR(B;p) &= \holim_{n,R} \thh(B)^{C_{p^n}}
\end{aligned}
\end{equation*}
are defined as the homotopy limits over the Frobenius and the restriction
maps
\begin{equation*}
F, R \: \thh(B)^{C_{p^n}} \to \thh(B)^{C_{p^{n-1}}} \,,
\end{equation*}
respectively.

\begin{corollary}\label{kkmodp-cor:6.10}
The comparison maps
\begin{equation*}
\begin{aligned}
\Gamma_n &\: \thh(\ell/p)^{C_{p^n}} \to \thh(\ell/p)^{hC_{p^n}} \\
\hat\Gamma_n &\: \thh(\ell/p)^{C_{p^{n-1}}} \to \thh(\ell/p)^{tC_{p^n}}
\end{aligned}
\end{equation*}
for $n\ge1$, and
\begin{equation*}
\begin{aligned}
\Gamma &\: TF(\ell/p;p) \to \thh(\ell/p)^{hS^1} \\
\hat\Gamma &\: TF(\ell/p;p) \to \thh(\ell/p)^{tS^1}
\end{aligned}
\end{equation*}
all induce $(2p-2)$-coconnected homomorphisms on $V(1)$-homotopy.
\end{corollary}

\begin{proof}
  This follows from a theorem of Tsalidis
  \cite{kkmodp-Ts98}*{Th.~2.4} and
Proposition~\ref{kkmodp-prop:6.9}
above, just like in \cite{kkmodp-AR02}*{Th.~5.7}.  See also
\cite{kkmodp-BBLNR}*{Ex.\,10.2}
\end{proof}

\section{Higher fixed points }
\label{kkmodp-sec:higher}
Let $n\ge1$.  Write $v_p(i)$ for the $p$-adic valuation of $i$.
Define a numerical function $\rho(-)$ by
\begin{equation*}
\begin{aligned}
\rho(2k-1) &= (p^{2k+1}+1)/(p+1) = p^{2k} - p^{2k-1} + \dots - p + 1 \\
\rho(2k) &= (p^{2k+2}-p^2)/(p^2-1) = p^{2k} + p^{2k-2} + \dots + p^2
\end{aligned}
\end{equation*}
for $k\ge0$, so $\rho(-1) = 1$ and $\rho(0) = 0$.  For even arguments,
$\rho(2k) = r(2k)$ as defined in \cite{kkmodp-AR02}*{Def.~2.5}.

In all of the following spectral sequences we know that $\lambda_2$,
$t\mu_2$ and $\bar\epsilon_1$ are infinite cycles.  For $\lambda_2$ and
$\bar\epsilon_1$ this follows from the $C_{p^n}$-fixed point analogue
of diagram~\eqref{kkmodp-eq:6.2}, by \cite{kkmodp-AR02}*{Prop.~2.8} and
Lemma~\ref{kkmodp-lem:6.4}.  For $t\mu_2$
it follows from \cite{kkmodp-AR02}*{Prop.~4.8}, by naturality.

\begin{theorem}\label{kkmodp-thm:7.1}
The $C_{p^n}$-Tate spectral sequence in
$V(1)$-homotopy for $\thh(\ell/p)$ begins
\begin{equation*}
\hat E^2(C_{p^n}, \ell/p) = E(u_n, \lambda_2) \otimes \F_p\{1,
  \epsilon_0, \mu_0, \epsilon_0 \mu_0, \dots, \mu_0^{p-1}, \bar\epsilon_1\}
  \otimes P(t^{\pm1}, \mu_2)
\end{equation*}
and converges to $V(1)_* \thh(\ell/p)^{tC_{p^n}}$.  It is a module
spectral sequence over the algebra spectral sequence $\hat E^*(C_{p^n},
\ell)$ converging to $V(1)_* \thh(\ell)^{tC_{p^n}}$.

There is an initial $d^2$-differential generated by
\begin{equation*}
d^2(\epsilon_0 \mu_0^{i-1}) = t \mu_0^i
\end{equation*}
for $0<i<p$.
Next, there are $2n$ families of even length differentials generated by
\begin{equation*}
d^{2\rho(2k-1)}(t^{p^{2k-1}-p^{2k}+i} \cdot \bar\epsilon_1)
	\doteq (t\mu_2)^{\rho(2k-3)} \cdot t^i
\end{equation*}
for $v_p(i) = 2k-2$, for each $k = 1, \dots, n$, and
\begin{equation*}
d^{2\rho(2k)}(t^{p^{2k-1}-p^{2k}})
	\doteq \lambda_2 \cdot t^{p^{2k-1}} \cdot (t\mu_2)^{\rho(2k-2)}
\end{equation*}
for each $k = 1, \dots, n$.
Finally, there is a differential of odd length generated by
\begin{equation*}
d^{2\rho(2n)+1}(u_n \cdot t^{-p^{2n}}) \doteq (t\mu_2)^{\rho(2n-2)+1} \,.
\end{equation*}
\end{theorem}

We shall prove Theorem~\ref{kkmodp-thm:7.1} by induction on $n$.  The base
case $n=1$ was covered by Proposition~\ref{kkmodp-prop:6.8}.  We can therefore
assume that Theorem~\ref{kkmodp-thm:7.1} holds for some fixed $n\ge1$, and
must prove the corresponding statement for $n+1$.  First we make the
following deduction.

\begin{corollary}\label{kkmodp-cor:7.2}
The initial differential in the $C_{p^n}$-Tate spectral sequence in
$V(1)$-ho\-mo\-to\-py for $\thh(\ell/p)$ leaves
\begin{equation*}
\hat E^3(C_{p^n}, \ell/p) = E(u_n, \bar\epsilon_1, \lambda_2)
	\otimes P(t^{\pm1}, t\mu_2) \,.
\end{equation*}
The next $2n$ families of differentials leave the intermediate terms
\begin{equation*}
\begin{aligned}
\hat E^{2\rho(1)+1}(C_{p^n}&, \ell/p) = E(u_n, \lambda_2)
	\otimes \F_p\{t^{-i} \mid 0<i<p\} \otimes P(t^{\pm p}) \\
&\qquad \oplus E(u_n, \bar\epsilon_1, \lambda_2) \otimes P(t^{\pm p},
	t\mu_2)
\end{aligned}
\end{equation*}
(for $m=1$),
\begin{equation*}
\begin{aligned}
\hat E^{2\rho(2m-1)+1}&(C_{p^n}, \ell/p) = E(u_n, \lambda_2)
	\otimes \F_p\{t^{-i} \mid 0<i<p\} \otimes P(t^{\pm p^2}) \\
&\quad \oplus \bigoplus_{k=2}^m E(u_n, \lambda_2) \otimes
	\F_p\{t^j \mid j\in\Z,\, v_p(j) = 2k-2\} \otimes P_{\rho(2k-3)}(t\mu_2) \\
&\quad \oplus \bigoplus_{k=2}^{m-1} E(u_n, \bar\epsilon_1) \otimes
	\F_p\{t^j \lambda_2 \mid j\in\Z,\, v_p(j) = 2k-1\} \otimes
	P_{\rho(2k-2)}(t\mu_2) \\
&\qquad \oplus E(u_n, \bar\epsilon_1, \lambda_2) \otimes P(t^{\pm
	p^{2m-1}}, t\mu_2)
\end{aligned}
\end{equation*}
for $m=2, \dots, n$, and
\begin{equation*}
\begin{aligned}
\hat E^{2\rho(2m)+1}&(C_{p^n}, \ell/p) = E(u_n, \lambda_2) \otimes
	\F_p\{t^{-i} \mid 0<i<p\} \otimes P(t^{\pm p^2}) \\
&\quad \oplus \bigoplus_{k=2}^m E(u_n, \lambda_2) \otimes
	\F_p\{t^j \mid j\in\Z,\, v_p(j) = 2k-2\} \otimes P_{\rho(2k-3)}(t\mu_2) \\
&\quad \oplus \bigoplus_{k=2}^m E(u_n, \bar\epsilon_1) \otimes
	\F_p\{t^j \lambda_2 \mid j\in\Z,\, v_p(j) = 2k-1\} \otimes
	P_{\rho(2k-2)}(t\mu_2) \\
&\qquad \oplus E(u_n, \bar\epsilon_1, \lambda_2) \otimes P(t^{\pm
	p^{2m}}, t\mu_2)
\end{aligned}
\end{equation*}
for $m = 1, \dots, n$.  The final differential leaves the
$E^{2\rho(2n)+2} = E^\infty$-term, equal to
\begin{equation*}
\begin{aligned}
\hat E^\infty(C_{p^n}, \ell/p) &= E(u_n, \lambda_2) \otimes
	\F_p\{t^{-i} \mid 0<i<p\} \otimes P(t^{\pm p^2})\\
&\quad \oplus \bigoplus_{k=2}^n E(u_n, \lambda_2) \otimes
	\F_p\{t^j \mid j\in\Z,\, v_p(j) = 2k-2\} \otimes P_{\rho(2k-3)}(t\mu_2) \\
&\quad \oplus \bigoplus_{k=2}^n E(u_n, \bar\epsilon_1) \otimes
	\F_p\{t^j \lambda_2 \mid j\in\Z,\, v_p(j) = 2k-1\} \otimes
	P_{\rho(2k-2)}(t\mu_2) \\
&\qquad \oplus E(\bar\epsilon_1, \lambda_2) \otimes P(t^{\pm
	p^{2n}}) \otimes P_{\rho(2n-2)+1}(t\mu_2) \,.
\end{aligned}
\end{equation*}
\end{corollary}

\begin{proof}
The statements about the $E^3$-, $E^{2\rho(1)+1}$- and
$E^{2\rho(2)+1}$-terms are clear from
Proposition~\ref{kkmodp-prop:6.8}.  For each $m =
2, \dots, n$ we proceed by a secondary induction.  The differential
\begin{equation*}
d^{2\rho(2m-1)}(t^{p^{2m-1}-p^{2m}+i} \cdot \bar\epsilon_1)
	\doteq (t\mu_2)^{\rho(2m-3)} \cdot t^i
\end{equation*}
for $v_p(i) = 2m-2$ is non-trivial only on the summand
\begin{equation*}
E(u_n, \bar\epsilon_1, \lambda_2) \otimes P(t^{\pm p^{2m-2}}, t\mu_2)
\end{equation*}
of the $E^{2\rho(2m-2)+1} = E^{2\rho(2m-1)}$-term, with homology
\begin{equation*}
\begin{aligned}
&E(u_n, \lambda_2) \otimes \F_p\{t^j \mid j\in\Z,\, v_p(j) = 2m-2\} \otimes
	P_{\rho(2m-3)}(t\mu_2) \\
&\quad \oplus E(u_n, \bar\epsilon_1, \lambda_2) \otimes P(t^{\pm
	p^{2m-1}}, t\mu_2) \,.
\end{aligned}
\end{equation*}
This gives the stated $E^{2\rho(2m-1)+1}$-term.  Similarly, the differential
\begin{equation*}
d^{2\rho(2m)}(t^{p^{2m-1}-p^{2m}})
	\doteq \lambda_2 \cdot t^{p^{2m-1}} \cdot (t\mu_2)^{\rho(2m-2)}
\end{equation*}
is non-trivial only on the summand
\begin{equation*}
E(u_n, \bar\epsilon_1, \lambda_2) \otimes P(t^{\pm p^{2m-1}}, t\mu_2)
\end{equation*}
of the $E^{2\rho(2m-1)+1} = E^{2\rho(2m)}$-term, with homology
\begin{equation*}
\begin{aligned}
&E(u_n, \bar\epsilon_1) \otimes \F_p\{t^j \lambda_2 \mid
j\in\Z,\, v_p(j) = 2m-1\}
	\otimes P_{\rho(2m-2)}(t\mu_2) \\
&\quad \oplus E(u_n, \bar\epsilon_1, \lambda_2) \otimes P(t^{\pm
	p^{2m}}, t\mu_2)
\,.
\end{aligned}
\end{equation*}
This gives the stated $E^{2\rho(2m)+1}$-term.  The final differential
\begin{equation*}
d^{2\rho(2n)+1}(u_n \cdot t^{-p^{2n}}) \doteq (t\mu_2)^{\rho(2n-2)+1}
\end{equation*}
is non-trivial only on the summand
\begin{equation*}
E(u_n, \bar\epsilon_1, \lambda_2) \otimes P(t^{\pm p^{2n}}, t\mu_2)
\end{equation*}
of the $E^{2\rho(2n)+1}$-term, with homology
\begin{equation*}
E(\bar\epsilon_1, \lambda_2) \otimes P(t^{\pm p^{2n}}) \otimes
	P_{\rho(2n-2)+1}(t\mu_2)
\,.
\end{equation*}
This gives the stated $E^{2\rho(2n)+2}$-term.  At this stage there is
no room for any further differentials, since the spectral sequence is
concentrated in a narrower horizontal band than the vertical height of
the following differentials.
\end{proof}

Next we compare the $C_{p^n}$-Tate spectral sequence with the
$C_{p^n}$-homotopy fixed point spectral sequence obtained by restricting the
$E^2$-term to the second quadrant ($s\le0$, $t\ge0$).  It is
algebraically easier to handle the latter after inverting $\mu_2$,
which can be interpreted as comparing $\thh(\ell/p)$ with its
$C_p$-Tate construction.

In general, there is a commutative diagram
\begin{equation}
\xymatrix{
\thh(B)^{C_{p^n}} \ar[r]^-R \ar[d]^{\Gamma_n}
  & \thh(B)^{C_{p^{n-1}}} \ar[r]^-{\Gamma_{n-1}} \ar[d]^{\hat\Gamma_n}
  & \thh(B)^{hC_{p^{n-1}}} \ar[d]^{\hat\Gamma_1^{hC_{p^{n-1}}}} \\
\thh(B)^{hC_{p^n}} \ar[r]^-{R^h}
  & \thh(B)^{tC_{p^n}} \ar[r]^-{G_{n-1}}
  & (\rho_p^* \thh(B)^{tC_p})^{hC_{p^{n-1}}} \,.
}
\label{kkmodp-eq:7.3}
\end{equation}
Here $\rho_p^* \thh(B)^{tC_p}$
is a notation for  the $S^1$-spectrum obtained from the
$S^1/C_p$-spectrum 
$\thh(B)^{tC_p}$ via the $p$-th root isomorphism
$\rho_p \: S^1\to S^1/C_p$, and $G_{n-1}$ is the comparison map from the
$C_{p^{n-1}}$-fixed points to the $C_{p^{n-1}}$-homotopy fixed points
of $\rho_p^* \thh(B)^{tC_p}$, in view of the identification
\begin{equation*}
	(\rho_p^*\thh(B)^{tC_p})^{C_{p^{n-1}}} = \thh(B)^{tC_{p^n}} \,.
\end{equation*}

We are of course considering the case $B = \ell/p$.  In $V(1)$-homotopy
all four maps with labels containing $\Gamma$ are $(2p-2)$-coconnected, by
Corollary~\ref{kkmodp-cor:6.10}, so $G_{n-1}$ is at least $(2p-1)$-coconnected.
(We shall see in Lemma~\ref{kkmodp-lem:7.11} that $V(1)_* G_{n-1}$ is an
isomorphism in all degrees.)
By Proposition~\ref{kkmodp-prop:6.9} the map $\hat\Gamma_1$ precisely inverts $\mu_2$, so
the $E^2$-term of the $C_{p^n}$-homotopy fixed point spectral sequence in
$V(1)$-homotopy for $\thh(\ell/p)^{tC_p}$ is obtained by inverting $\mu_2$
in $E^2(C_{p^n}, \ell/p)$.  We denote this
spectral sequence by $\mu_2^{-1} E^*(C_{p^n},
\ell/p)$, even though in later terms only a power of $\mu_2$ is present.

\begin{theorem}\label{kkmodp-thm:7.4}
The $C_{p^n}$-homotopy fixed point spectral sequence 
$\mu_2^{-1}E^*(C_{p^n}, \ell/p)$ in 
$V(1)$-homotopy for $\thh(\ell/p)^{tC_p}$ begins
\begin{equation*}
\mu_2^{-1} E^2(C_{p^n}, \ell/p) = E(u_n, \lambda_2) \otimes \F_p\{1,
  \epsilon_0, \mu_0, \epsilon_0 \mu_0, \dots, \mu_0^{p-1}, \bar\epsilon_1\}
  \otimes P(t, \mu_2^{\pm1})
\end{equation*}
and converges to $V(1)_* (\rho_p^* \thh(\ell/p)^{tC_p})^{hC_{p^n}}$, which
receives a $(2p-2)$-coconnected map $(\hat\Gamma_1)^{hC_{p^n}}$ from
$V(1)_* \thh(\ell/p)^{hC_{p^n}}$.
There is an initial $d^2$-differential generated by
\begin{equation*}
d^2(\epsilon_0 \mu_0^{i-1}) = t \mu_0^i
\end{equation*}
for $0<i<p$.
Next, there are $2n$ families of even length differentials generated by
\begin{equation*}
d^{2\rho(2k-1)}(\mu_2^{p^{2k}-p^{2k-1}+j} \cdot \bar\epsilon_1)
	\doteq (t\mu_2)^{\rho(2k-1)} \cdot \mu_2^j
\end{equation*}
for $v_p(j) = 2k-2$, for each $k = 1, \dots, n$, and
\begin{equation*}
d^{2\rho(2k)}(\mu_2^{p^{2k}-p^{2k-1}})
	\doteq \lambda_2 \cdot \mu_2^{-p^{2k-1}} \cdot (t\mu_2)^{\rho(2k)}
\end{equation*}
for each $k = 1, \dots, n$.
Finally, there is a differential of odd length generated by
\begin{equation*}
d^{2\rho(2n)+1}(u_n \cdot \mu_2^{p^{2n}}) \doteq (t\mu_2)^{\rho(2n)+1} \,.
\end{equation*}
\end{theorem}

\begin{proof}
The differential pattern follows from Theorem~\ref{kkmodp-thm:7.1} by naturality with
respect to the maps of spectral sequences
\begin{equation*}
\mu_2^{-1} E^*(C_{p^n}, \ell/p)
\xl{\hat\Gamma_1^{hC_{p^n}}}
E^*(C_{p^n}, \ell/p)
\xr{R^h}
\hat E^*(C_{p^n}, \ell/p)
\end{equation*}
induced by $\hat\Gamma_1^{hC_{p^n}}$ and $R^h$.  The first inverts
$\mu_2$ and the second inverts $t$, at the level of $E^2$-terms.
We are also using that $t\mu_2$, the image of $v_2$, multiplies as an
infinite cycle in all of these spectral sequences.
\end{proof}

\begin{corollary}\label{kkmodp-cor:7.5}
The initial differential in the $C_{p^n}$-homotopy fixed point spectral
sequence in $V(1)$-homotopy for $\thh(\ell/p)^{tC_p}$ leaves
\begin{equation*}
\begin{aligned}
\mu_2^{-1} E^3(C_{p^n}, \ell/p) &= E(u_n, \lambda_2) \otimes
	\F_p\{\mu_0^i \mid 0<i<p\} \otimes P(\mu_2^{\pm1}) \\
&\qquad \oplus E(u_n, \bar\epsilon_1, \lambda_2)
        \otimes P(\mu_2^{\pm1}, t\mu_2) \,.
\end{aligned}
\end{equation*}
The next $2n$ families of differentials leave the intermediate terms
\begin{equation*}
\begin{aligned}
\mu_2^{-1} E^{2\rho(2m-1)+1}&(C_{p^n}, \ell/p) = E(u_n, \lambda_2)
	\otimes \F_p\{\mu_0^i \mid 0<i<p\} \otimes P(\mu_2^{\pm1}) \\
&\oplus \bigoplus_{k=1}^m E(u_n, \lambda_2) \otimes
        \F_p\{\mu_2^j \mid j\in\Z,\, v_p(j) = 2k-2\} \otimes P_{\rho(2k-1)}(t\mu_2) \\
&\oplus \bigoplus_{k=1}^{m-1} E(u_n, \bar\epsilon_1) \otimes
        \F_p\{\lambda_2 \mu_2^j \mid j\in\Z,\, v_p(j) = 2k-1\} \otimes
        P_{\rho(2k)}(t\mu_2) \\
&\quad \oplus E(u_n, \bar\epsilon_1, \lambda_2) \otimes P(\mu_2^{\pm
        p^{2m-1}}, t\mu_2)
\end{aligned}
\end{equation*}
and
\begin{equation*}
\begin{aligned}
\mu_2^{-1} E^{2\rho(2m)+1}&(C_{p^n}, \ell/p) = E(u_n, \lambda_2)
	\otimes \F_p\{\mu_0^i \mid 0<i<p\} \otimes P(\mu_2^{\pm1}) \\
&\oplus \bigoplus_{k=1}^m E(u_n, \lambda_2) \otimes
        \F_p\{\mu_2^j \mid j\in\Z,\, v_p(j) = 2k-2\} \otimes P_{\rho(2k-1)}(t\mu_2) \\
&\oplus \bigoplus_{k=1}^m E(u_n, \bar\epsilon_1) \otimes
	\F_p\{\lambda_2 \mu_2^j \mid j\in\Z,\, v_p(j) = 2k-1\} \otimes
	P_{\rho(2k)}(t\mu_2) \\
&\quad \oplus E(u_n, \bar\epsilon_1, \lambda_2) \otimes P(\mu_2^{\pm
        p^{2m}}, t\mu_2)
\end{aligned}
\end{equation*}
for $m = 1, \dots, n$.  The final differential leaves the
$E^{2\rho(2n)+2} = E^\infty$-term, equal to
\begin{equation*}
\begin{aligned}
\mu_2^{-1} E^\infty&(C_{p^n}, \ell/p) = E(u_n, \lambda_2) \otimes
	\F_p\{\mu_0^i \mid 0<i<p\} \otimes P(\mu_2^{\pm1}) \\
&\quad \oplus \bigoplus_{k=1}^n E(u_n, \lambda_2) \otimes
	\F_p\{\mu_2^j \mid j\in\Z,\, v_p(j) = 2k-2\} \otimes P_{\rho(2k-1)}(t\mu_2) \\
&\quad \oplus \bigoplus_{k=1}^n E(u_n, \bar\epsilon_1) \otimes
	\F_p\{\lambda_2 \mu_2^j \mid j\in\Z,\, v_p(j) = 2k-1\} \otimes
	P_{\rho(2k)}(t\mu_2) \\
&\qquad \oplus E(\bar\epsilon_1, \lambda_2) \otimes P(\mu_2^{\pm
	p^{2n}}) \otimes P_{\rho(2n)+1}(t\mu_2) \,.
\end{aligned}
\end{equation*}
\end{corollary}

\begin{proof}
The computation of the $E^3$-term from the $E^2$-term is straightforward.
The rest of the proof goes by a secondary induction on $m = 1, \dots,
n$, very much like the proof of Corollary~\ref{kkmodp-cor:7.2}.  The differential
\begin{equation*}
d^{2\rho(2m-1)}(\mu_2^{p^{2m}-p^{2m-1}+j} \cdot \bar\epsilon_1)
	\doteq (t\mu_2)^{\rho(2m-1)} \cdot \mu_2^j
\end{equation*}
for $v_p(j) = 2m-2$ is non-trivial only on the summand
\begin{equation*}
E(u_n, \bar\epsilon_1, \lambda_2) \otimes P(\mu_2^{\pm p^{2m-2}}, t\mu_2)
\end{equation*}
of the $E^3 = E^{2\rho(1)}$-term (for $m=1$), resp.~the
$E^{2\rho(2m-2)+1} = E^{2\rho(2m-1)}$-term (for $m = 2, \dots, n$).
Its homology is
\begin{equation*}
\begin{aligned}
&E(u_n, \lambda_2) \otimes \F_p\{ \mu_2^j \mid j\in\Z,\, v_p(j) = 2m-2 \} \otimes
	P_{\rho(2m-1)}(t\mu_2) \\
&\quad\oplus E(u_n, \bar\epsilon_1, \lambda_2) \otimes P(\mu_2^{\pm
	p^{2m-1}}, t\mu_2) \,,
\end{aligned}
\end{equation*}
which gives the stated $E^{2\rho(2m-1)+1}$-term.
The differential
\begin{equation*}
d^{2\rho(2m)}(\mu_2^{p^{2m}-p^{2m-1}}) \doteq \lambda_2 \cdot \mu_2^{-p^{2m-1}}
	\cdot (t\mu_2)^{\rho(2m)}
\end{equation*}
is non-trivial only on the summand
\begin{equation*}
E(u_n, \bar\epsilon_1, \lambda_2) \otimes P(\mu_2^{\pm p^{2m-1}}, t\mu_2)
\end{equation*}
of the $E^{2\rho(2m-1)+1} = E^{2\rho(2m)}$-term, leaving
\begin{equation*}
\begin{aligned}
&E(u_n, \bar\epsilon_1) \otimes \F_p\{ \lambda_2 \mu_2^j
        \mid j\in\Z,\, v_p(j) = 2m-1 \} \otimes P_{\rho(2m)}(t\mu_2) \\
&\quad\oplus E(u_n, \bar\epsilon_1, \lambda_2) \otimes P(\mu_2^{\pm p^{2m}},
        t\mu_2) \,.
\end{aligned}
\end{equation*}
This gives the stated $E^{2\rho(2m)+1}$-term.
The final differential
\begin{equation*}
d^{2\rho(2n)+1}(u_n \cdot \mu_2^{p^{2n}}) \doteq (t\mu_2)^{\rho(2n)+1}
\end{equation*}
is non-trivial only on the summand
\begin{equation*}
E(u_n, \bar\epsilon_1, \lambda_2) \otimes P(\mu_2^{\pm p^{2n}}, t\mu_2)
\end{equation*}
of the $E^{2\rho(2n)+1}$-term, with homology
\begin{equation*}
E(\bar\epsilon_1, \lambda_2) \otimes P(\mu_2^{\pm p^{2n}}) \otimes
        P_{\rho(2n)+1}(t\mu_2)\,.
\end{equation*}
This gives the stated $E^{2\rho(2n)+2}$-term.  There is
no room for any further differentials, since the spectral sequence is
concentrated in a narrower vertical band than the horizontal width of
the following differentials, so $E^{2\rho(2n)+2} = E^\infty$.
\end{proof}

\begin{proof}[Proof of Theorem~\ref{kkmodp-thm:7.1}]
To make the inductive step to $C_{p^{n+1}}$, we use that
the first $d^r$-differential of odd length in $\hat E^*(C_{p^n}, \ell/p)$
occurs for $r = r_0 = 2\rho(2n)+1$.  It follows from
\cite{kkmodp-AR02}*{Lem.~5.2}
that the terms $\hat E^r(C_{p^n}, \ell/p)$ and $\hat E^r(C_{p^{n+1}},
\ell/p)$ are isomorphic for $r \le 2\rho(2n)+1$, via the Frobenius
map (taking $t^i$ to $t^i$) in even columns and the Verschiebung map
(taking $u_n t^i$ to $u_{n+1} t^i$) in odd columns.  Furthermore, the
differential $d^{2\rho(2n)+1}$ is zero in the latter spectral sequence.
This proves the part of Theorem~\ref{kkmodp-thm:7.1} for $n+1$ that concerns the
differentials leading up to the term
\begin{equation}
\begin{aligned}
\hat E&^{2\rho(2n)+2}(C_{p^{n+1}}, \ell/p) = E(u_{n+1}, \lambda_2) \otimes
        \F_p\{t^{-i} \mid 0<i<p\} \otimes P(t^{\pm p^2}) \\
&\quad \oplus \bigoplus_{k=2}^n E(u_{n+1}, \lambda_2) \otimes
        \F_p\{t^j \mid j\in\Z,\, v_p(j) = 2k-2\} \otimes P_{\rho(2k-3)}(t\mu_2) \\
&\quad \oplus \bigoplus_{k=2}^n E(u_{n+1}, \bar\epsilon_1) \otimes
        \F_p\{t^j \lambda_2 \mid j\in\Z,\, v_p(j) = 2k-1\} \otimes
        P_{\rho(2k-2)}(t\mu_2) \\
&\qquad \oplus E(u_{n+1}, \bar\epsilon_1, \lambda_2) \otimes P(t^{\pm
        p^{2n}}, t\mu_2) \,.
\end{aligned}
\label{kkmodp-eq:7.6}
\end{equation}

Next we use the following commutative diagram, where we abbreviate
$\thh(B)$ to $T(B)$ for typographical reasons:
\begin{equation}
\xymatrix{
(\rho_p^* T(B)^{tC_p})^{hC_{p^n}} \ar[d]^F
& T(B)^{hC_{p^n}} \ar[l]_-{\hat\Gamma_1^{hC_{p^n}}} \ar[d]^F
& T(B)^{C_{p^n}} \ar[l]_{\Gamma_n} \ar[r]^-{\hat\Gamma_{n+1}} \ar[d]^F
& T(B)^{tC_{p^{n+1}}} \ar[d]^F \\
\rho_p^* T(B)^{tC_p}
& T(B) \ar[l]_{\hat\Gamma_1}
& T(B) \ar@{=}[l] \ar[r]^{\hat\Gamma_1}
& \rho_p^* T(B)^{tC_p}\,.
}
\label{kkmodp-eq:7.7}
\end{equation}
The horizontal maps all induce $(2p-2)$-coconnected maps in
$V(1)$-homotopy for $B = \ell/p$.  Each $F$ is a Frobenius map,
forgetting invariance under a $C_{p^n}$-action.
Thus the map $\hat\Gamma_{n+1}$ to
the right induces an isomorphism of $E(\lambda_2) \otimes P(v_2)$-modules
in all degrees $* > (2p-2)$ from $V(1)_* \thh(\ell/p)^{C_{p^n}}$, implicitly
identified to the left with the abutment of $\mu_2^{-1} E^*(C_{p^n},
\ell/p)$, to $V(1)_* \thh(\ell/p)^{tC_{p^{n+1}}}$, which is the abutment
of $\hat E^*(C_{p^{n+1}}, \ell/p)$.  The diagram above ensures that the
isomorphism induced by $\hat\Gamma_{n+1}$ is compatible with the one
induced by $\hat\Gamma_1$.  By Proposition~\ref{kkmodp-prop:6.9} 
it takes $\bar\epsilon_1$,
$\lambda_2$ and $\mu_2$ to $\bar\epsilon_1$, $\lambda_2$ and $t^{-p^2}$
up to a unit factor in $\F_p$,
respectively, and similarly for monomials in these classes.

We focus on the summand
\begin{equation*}
E(u_n, \lambda_2) \otimes \F_p\{\mu_2^j \mid j\in\Z,\, v_p(j) = 2n-2\}
\otimes P_{\rho(2n-1)}(t\mu_2)
\end{equation*}
in $\mu_2^{-1} E^\infty(C_{p^n}, \ell/p)$, abutting to $V(1)_*
\thh(\ell/p)^{C_{p^n}}$ in degrees $> (2p-2)$.  In the $P(v_2)$-module
structure on the abutment, each class $\mu_2^j$ with $v_p(j) = 2n-2$,
$j>0$, generates a copy of $P_{\rho(2n-1)}(v_2)$, since there are no
permanent cycles in the same total degree as $y = (t\mu_2)^{\rho(2n-1)}
\cdot \mu_2^j$ that have lower (= more negative) homotopy fixed point
filtration.  See Lemma~\ref{kkmodp-lem:7.8} below for the elementary verification.  The
$P(v_2)$-module isomorphism induced by $\hat\Gamma_{n+1}$ must take this
to a copy of $P_{\rho(2n-1)}(v_2)$ in $V(1)_* \thh(\ell/p)^{tC_{p^{n+1}}}$,
generated by $t^{-p^2j}$.

Writing $i = -p^2j$, we deduce that for $v_p(i) = 2n$, $i<0$, the
infinite cycle $z = (t\mu_2)^{\rho(2n-1)} \cdot t^i$ must represent zero
in the abutment, and must therefore be hit by a differential $z = d^r(x)$
in the $C_{p^{n+1}}$-Tate spectral sequence.  Here $r \ge 2\rho(2n)+2$.

Since $z$ generates a free copy of $P(t\mu_2)$ in the
$E^{2\rho(2n)+2}$-term displayed in~\eqref{kkmodp-eq:7.6}, and $d^r$ is
$P(t\mu_2)$-linear, the class $x$ cannot be annihilated by any power
of $t\mu_2$.  This means that $x$ must be contained in the summand
\begin{equation*}
E(u_{n+1}, \bar\epsilon_1, \lambda_2) \otimes P(t^{\pm p^{2n}}, t\mu_2)
\end{equation*}
of $\hat E^{2\rho(2n)+2}(C_{p^{n+1}}, \ell/p)$.  By an elementary
check of bidegrees, see Lemma~\ref{kkmodp-lem:7.9} below, the only possibility is that $x$
has vertical degree $(2p-1)$, so that we have differentials
\begin{equation*}
d^{2\rho(2n+1)}(t^{p^{2n+1}-p^{2n+2}+i} \cdot \bar\epsilon_1)
	\doteq (t\mu_2)^{\rho(2n-1)} \cdot t^i
\end{equation*}
for all $i<0$ with $v_p(i) = 2n$.  The cases $i>0$ follow by the module
structure over the $C_{p^{n+1}}$-Tate spectral sequence for $\ell$.
The remaining two differentials,
\begin{equation*}
d^{2\rho(2n+2)}(t^{p^{2n+1}-p^{2n+2}})
        \doteq \lambda_2 \cdot t^{p^{2n+1}} \cdot (t\mu_2)^{\rho(2n)}
\end{equation*}
and
\begin{equation*}
d^{2\rho(2n+2)+1}(u_{n+1} \cdot t^{-p^{2n+2}}) \doteq (t\mu_2)^{\rho(2n)+1}
\end{equation*}
are also present in the $C_{p^{n+1}}$-Tate spectral sequence for
$\ell$, see \cite{kkmodp-AR02}*{Th.~6.1}, hence follow in the present case by the
module structure.  With this we have established the complete
differential pattern asserted by Theorem~\ref{kkmodp-thm:7.1}.
\end{proof}

\begin{lemma}\label{kkmodp-lem:7.8}
For $j\in\Z$ with $v_p(j) = 2n-2$, where $n\ge1$, there are no classes in $\mu_2^{-1}
E^\infty(C_{p^n}, \ell/p)$ in the same total degree as
$y = (t\mu_2)^{\rho(2n-1)} \cdot \mu_2^j$ that have lower homotopy fixed
point filtration.
\end{lemma}

\begin{proof}
The total degree of $y$ is $2(p^{2n+2} - p^{2n+1} + p - 1) + 2p^2 j \equiv
(2p-2) \mod 2p^{2n}$, which is even.

Looking at the formula for $\mu_2^{-1} E^\infty(C_{p^n}, \ell/p)$
in Corollary~\ref{kkmodp-cor:7.5}, the classes of lower filtration than $y$ all lie in
the terms
\begin{equation*}
E(u_n, \bar\epsilon_1) \otimes \F_p\{ \lambda_2 \mu_2^i \mid
j\in\Z,\, v_p(i)=2n-1 \}
  \otimes P_{\rho(2n)}(t\mu_2)
\end{equation*}
and
\begin{equation*}
E(\bar\epsilon_1, \lambda_2) \otimes P(\mu_2^{\pm p^{2n}}) \otimes
  P_{\rho(2n)+1}(t\mu_2) \,.
\end{equation*}
Those in even total degree and of lower filtration than $y$ are
\begin{equation*}
u_n \lambda_2 \cdot \mu_2^i (t\mu_2)^e, \quad
\bar\epsilon_1 \lambda_2 \cdot \mu_2^i (t\mu_2)^e
\end{equation*}
with $v_p(i) = 2n-1$, $\rho(2n-1) < e < \rho(2n)$,
and
\begin{equation*}
\mu_2^i (t\mu_2)^e, \quad
\bar\epsilon_1 \lambda_2 \cdot \mu_2^i (t\mu_2)^e
\end{equation*}
with $v_p(i) \ge 2n$, $\rho(2n-1) < e \le \rho(2n)$.

The total degree of $u_n \lambda_2 \cdot \mu_2^i (t\mu_2)^e$ for $v_p(i)
= 2n-1$ is $(-1) + (2p^2-1) + 2p^2 i + (2p^2-2)e \equiv (2p^2-2)(e+1)
\mod 2p^{2n}$.  For this to agree with the total degree of $y$, we must
have $(2p-2) \equiv (2p^2-2)(e+1) \mod 2p^{2n}$, so $(e+1) \equiv 1/(1+p)
\mod p^{2n}$ and $e \equiv \rho(2n-1) - 1 \mod p^{2n}$.  There is no
such $e$ with $\rho(2n-1) < e < \rho(2n)$.

The total degree of $\bar\epsilon_1 \lambda_2 \cdot \mu_2^i (t\mu_2)^e$
for $v_p(i) = 2n-1$ is $(2p-1) + (2p^2-1) + 2p^2 i + (2p^2-2)e \equiv 2p
+ (2p^2-2)(e+1) \mod 2p^{2n}$.  To agree with that
of $y$, we must have $(2p-2) \equiv 2p + (2p^2-2)(e+1) \mod 2p^{2n}$, so
$(e+1) \equiv 1/(1-p^2) \mod p^{2n}$ and $e \equiv \rho(2n) \mod p^{2n}$.
There is no such $e$ with $\rho(2n-1) < e < \rho(2n)$.

The total degree of $\mu_2^i (t\mu_2)^e$ for $v_p(i) \ge 2n$ is $2p^2 i +
(2p^2-2)e \equiv (2p^2-2)e \mod 2p^{2n}$.  To agree with that of $y$, we
must have $(2p-2) \equiv (2p^2-2)e \mod 2p^{2n}$, so $e \equiv 1/(1+p)
\equiv \rho(2n-1) \mod p^{2n}$.  There is no such $e$ with $\rho(2n-1)
< e \le \rho(2n)$.

The total degree of $\bar\epsilon_1 \lambda_2 \cdot \mu_2^i (t\mu_2)^e$
for $v_p(i) \ge 2n$ is $(2p-1) + (2p^2-1) + 2p^2 i + (2p^2-2)e$.
To agree modulo~$2p^{2n}$ with that of $y$, we must have $e \equiv
\rho(2n) \mod p^{2n}$.  The only such $e$ with $\rho(2n-1) < e \le
\rho(2n)$ is $e = \rho(2n)$.  But in that case, the total degree of
$\bar\epsilon_1 \lambda_2 \cdot \mu_2^i (t\mu_2)^e$ is $2p + 2p^2 i +
(2p^2-2)(\rho(2n)+1) = 2(p^{2n+2} + p - 1) + 2p^2 i$.  To be equal to that
of $y$, we must have $2p^2 i + 2p^{2n+1} = 2p^2 j$, which is impossible
for $v_p(i) \ge 2n$ and $v_p(j) = 2n-2$.
\end{proof}

\begin{lemma}\label{kkmodp-lem:7.9}
For $v_p(i) = 2n$, $n\ge1$ and $z = (t\mu_2)^{\rho(2n-1)} \cdot t^i$,
the only class in
\begin{equation*}
E(u_{n+1}, \bar\epsilon_1, \lambda_2) \otimes P(t^{\pm p^{2n}}, t\mu_2)
\end{equation*}
that can support a non-zero differential $d^r(x) = z$
for $r \ge 2\rho(2n)+2$
is (a unit times)
\begin{equation*}
x = t^{p^{2n+1} - p^{2n+2} + i} \cdot \bar\epsilon_1 \,.
\end{equation*}
\end{lemma}

\begin{proof}
The class $z$ has total degree $(2p^2-2)\rho(2n-1) - 2i = 2p^{2n+2}
- 2p^{2n+1} + 2p - 2 - 2i \equiv (2p-2) \mod 2p^{2n}$, which is even,
and vertical degree $2p^2 \rho(2n-1)$.  Hence $x$ has odd total degree,
and vertical degree at most $2p^2 \rho(2n-1) - 2\rho(2n) - 1 = 2p^{2n+2}
- 2p^{2n+1} - \dots - 2p^3 - 1$.  This leaves the possibilities
\begin{equation*}
u_{n+1} \cdot t^j (t\mu_2)^e, \quad
\bar\epsilon_1 \cdot t^j (t\mu_2)^e, \quad
\lambda_2 \cdot t^j (t\mu_2)^e
\end{equation*}
with $v_p(j) \ge 2n$ and $0 \le e < p^{2n} - p^{2n-1} - \dots - p
= \rho(2n-1) - \rho(2n-2) - 1$,
and
\begin{equation*}
u_{n+1} \bar\epsilon_1 \lambda_2 \cdot t^j (t\mu_2)^e
\end{equation*}
with $v_p(j) \ge 2n$ and $0 \le e < p^{2n} - p^{2n-1} - \dots - p - 1
= \rho(2n-1) - \rho(2n-2) - 2$.

The total degree of $x$ must be one more than the total degree of $z$,
hence is congruent to $(2p-1)$ modulo~$2p^{2n}$.

The total degree of $u_{n+1} \cdot t^j (t\mu_2)^e$ is $-1 - 2j +
(2p^2-2)e \equiv -1 + (2p^2-2)e \mod 2p^{2n}$.  To have $(2p-1) \equiv -1
+ (2p^2-2)e \mod 2p^{2n}$ we must have $e \equiv -p/(1-p^2) \equiv p^{2n}
- p^{2n-1} - \dots - p \mod p^{2n}$, which does not happen for $e$
in the allowable range.

The total degree of $\lambda_2 \cdot t^j (t\mu_2)^e$ is $(2p^2-1) - 2j +
(2p^2-2)e \equiv (2p^2-1) + (2p^2-2)e \mod 2p^{2n}$.  To have $(2p-1)
\equiv (2p^2-1) + (2p^2-2)e \mod 2p^{2n}$ we must have $e \equiv -p/(1+p)
\equiv \rho(2n-1) - 1 \mod p^{2n}$, which does not happen.

The total degree of $u_{n+1} \bar\epsilon_1 \lambda_2 \cdot t^j
(t\mu_2)^e$ is $-1 + (2p-1) + (2p^2-1) - 2j + (2p^2-2)e \equiv (2p -
1) + (2p^2-2)(e+1) \mod 2p^{2n}$.  To have $(2p-1) \equiv (2p - 1) +
(2p^2-2)(e+1) \mod 2p^{2n}$ we must have $(e+1) \equiv 0 \mod p^{2n}$,
so $e \equiv p^{2n} - 1 \mod p^{2n}$, which does not happen.

The total degree of $\bar\epsilon_1 \cdot t^j (t\mu_2)^e$ is $(2p-1) -
2j + (2p^2-2)e \equiv (2p-1) + (2p^2-2)e \mod 2p^{2n}$.  To have $(2p-1)
\equiv (2p-1) + (2p^2-2)e \mod 2p^{2n}$, we must have $e \equiv 0 \mod
p^{2n}$, so $e = 0$ is the only possibility in the allowable range.
In that case, a check of total degrees shows that we must have $j =
p^{2n+1} - p^{2n+2} + i$.
\end{proof}

\begin{corollary}\label{kkmodp-cor:7.10}
$V(1)_* \thh(\ell/p)^{C_{p^n}}$ is finite in each degree.
\end{corollary}

\begin{proof}
This is clear by inspection of the $E^\infty$-term in
Corollary~\ref{kkmodp-cor:7.2}.
\end{proof}

\begin{lemma}\label{kkmodp-lem:7.11}
The map $G_n$ induces an isomorphism
\begin{equation*}
V(1)_* \thh(\ell/p)^{tC_{p^{n+1}}} \xr{\cong}
V(1)_* (\rho_p^*\thh(\ell/p)^{tC_p})^{hC_{p^n}}
\end{equation*}
in all degrees.  In the limit over the Frobenius maps $F$,
there is a map $G$ inducing an isomorphism
\begin{equation}\label{kkmodp-eq:6.11}
V(1)_* \thh(\ell/p)^{tS^1} \xr{\cong}
V(1)_* (\rho_p^*\thh(\ell/p)^{tC_p})^{hS^1} \,.
\end{equation}
\end{lemma}

\begin{proof}
As remarked after diagram~\eqref{kkmodp-eq:7.3}, $G_n$ induces an isomorphism
in $V(1)$-homotopy above degree $(2p-2)$.  The permanent cycle
$t^{-p^{2n+2}}$ in $\hat E^\infty(C_{p^{n+1}}, \ell)$ acts
invertibly on $\hat E^\infty(C_{p^{n+1}}, \ell/p)$, and
its image $G_n(t^{-p^{2n+2}}) = \mu_2^{p^{2n}}$ in $\mu_2^{-1}
E^\infty(C_{p^n}, \ell)$ acts invertibly on $\mu_2^{-1}
E^\infty(C_{p^n}, \ell/p)$.  Therefore the module action derived
from the $\ell$-algebra structure on $\ell/p$ ensures that $G_n$ induces
isomorphisms in $V(1)$-homotopy in all degrees.
\end{proof}

\begin{theorem}\label{kkmodp-thm:7.12}
The isomorphism~\eqref{kkmodp-eq:6.11} admits the following description
at the associated graded level:
\begin{itemize}
  \item[\textup{(a)}] 
The associated graded of $V(1)_* \thh(\ell/p)^{tS^1}$ for the
$S^1$-Tate spectral sequence is
\begin{equation*}
\begin{aligned}
\hat E^\infty&(S^1, \ell/p) = E(\lambda_2) \otimes
        \F_p\{t^{-i} \mid 0<i<p\} \otimes P(t^{\pm p^2})\\
&\quad \oplus \bigoplus_{k\ge2} E(\lambda_2) \otimes
        \F_p\{t^j \mid j\in\Z,\, v_p(j) = 2k-2\} \otimes P_{\rho(2k-3)}(t\mu_2) \\
&\quad \oplus \bigoplus_{k\ge2} E(\bar\epsilon_1) \otimes
        \F_p\{t^j \lambda_2 \mid j\in\Z,\, v_p(j) = 2k-1\} \otimes
        P_{\rho(2k-2)}(t\mu_2) \\
&\qquad \oplus E(\bar\epsilon_1, \lambda_2) \otimes P(t\mu_2) \,.
\end{aligned}
\end{equation*}
  \item[\textup{(b)}] 
The associated graded of\, $V(1)_* \thh(\ell/p)^{hS^1}$\, for the
$S^1$-homotopy fixed point spectral sequence maps by a
$(2p-2)$-coconnected map to
\begin{equation*}
\begin{aligned}
\mu_2^{-1} E^\infty&(S^1, \ell/p) = E(\lambda_2) \otimes
        \F_p\{\mu_0^i \mid 0<i<p\} \otimes P(\mu_2^{\pm1}) \\
&\quad \oplus \bigoplus_{k\ge1} E(\lambda_2) \otimes
        \F_p\{\mu_2^j \mid j\in\Z,\, v_p(j) = 2k-2\} \otimes P_{\rho(2k-1)}(t\mu_2) \\
&\quad \oplus \bigoplus_{k\ge1} E(\bar\epsilon_1) \otimes
        \F_p\{\lambda_2 \mu_2^j \mid j\in\Z,\, v_p(j) = 2k-1\} \otimes
        P_{\rho(2k)}(t\mu_2) \\
&\qquad \oplus E(\bar\epsilon_1, \lambda_2) \otimes P(t\mu_2) \,.
\end{aligned}
\end{equation*}
  \item[\textup{(c)}] 
    The isomorphism from~\textup{(a)} to~\textup{(b)} induced by $G$ takes $t^{-i}$ to
$\mu_0^i$ for $0<i<p$ and $t^i$ to $\mu_2^j$ for $i + p^2j =
0$, up to a unit factor in $\F_p$.
Furthermore, it takes multiples by $\bar\epsilon_1$, $\lambda_2$ or
$t\mu_2$ in the source to the same multiples in the target,
up to a unit factor in~$\F_p$.
\end{itemize}
\end{theorem}

\begin{proof}
Claims~(a) and~(b) follow by passage to the limit over $n$ from
Corollaries~\ref{kkmodp-cor:7.2} and~\ref{kkmodp-cor:7.5}.  Claim~(c) follows by passage to the same limit
from the formulas for the isomorphism induced by $\hat\Gamma_{n+1}$,
which were given below diagram~\eqref{kkmodp-eq:7.7}.
\end{proof}

\section{Topological cyclic homology }
\label{kkmodp-sec:tc}
By definition, there is a fiber sequence
\begin{equation*}
TC(B;p) \xr{\pi} TF(B;p) \xr{R-1} TF(B;p) \to \Sigma TC(B;p)
\end{equation*}
inducing a long exact sequence
\begin{equation}
\dots \xr{\partial} V(1)_* TC(B;p) \xr{\pi} V(1)_* TF(B;p) \xr{R_*-1}
V(1)_* TF(B;p) \xr{\partial} \dots
\label{kkmodp-eq:8.1}
\end{equation}
in $V(1)$-homotopy.  By Corollary~\ref{kkmodp-cor:6.10}, there are $(2p-2)$-coconnected
maps $\Gamma$ and $\hat\Gamma$ from $V(1)_* TF(\ell/p;p)$ to $V(1)_*
\thh(\ell/p)^{hS^1}$ and $V(1)_* \thh(\ell/p)^{tS^1}$, respectively.  We
model $V(1)_* TF(\ell/p;p)$ in degrees $> (2p-2)$ by the map $\hat\Gamma$
to the $S^1$-Tate construction.  Then, by diagram~\eqref{kkmodp-eq:7.3},
$R_*$ is modeled in the same range of degrees by the chain
of maps below\,:
\begin{equation*}
\xymatrix{
V(1)_* \thh(B)^{tS^1} \ar[dr]^G &
V(1)_* \thh(B)^{hS^1} \ar[d]^{(\hat\Gamma_1)^{hS^1}} \ar[r]^{R^h_*} &
V(1)_* \thh(B)^{tS^1} \\
& V(1)_* (\rho_p^*\thh(B)^{tC_p})^{hS^1}\,.
}
\end{equation*}
Here $R^h$ induces a map of spectral sequences
\begin{equation*}
E^*(R^h) \: E^*(S^1, B) \to \hat E^*(S^1, B)
\end{equation*}
(abutting to $R^h_*$),
which at the $E^2$-term equals the inclusion that algebraically inverts
$t$.  When $B = \ell/p$, the left hand map $G$ is an isomorphism by
Lemma~\ref{kkmodp-lem:7.11}, and the middle (wrong-way) map is $(2p-2)$-coconnected.

\begin{proposition}\label{kkmodp-prop:8.2}
In degrees $> (2p-2)$, the homomorphism
\begin{equation*}
E^\infty(R^h) \: E^\infty(S^1, \ell/p) \to \hat E^\infty(S^1, \ell/p)
\end{equation*}
maps
\begin{itemize}
  \item[\textup{(a)}]$E(\bar\epsilon_1, \lambda_2) \otimes P(t\mu_2)$ identically to the
same expression;
\item[\textup{(b)}]$E(\lambda_2) \otimes \F_p\{\mu_2^{-j}\} \otimes P_{\rho(2k-1)}(t\mu_2)$
surjectively onto
$E(\lambda_2) \otimes \F_p\{t^j\} \otimes P_{\rho(2k-3)}(t\mu_2)$
for each $k\ge2$, $j = dp^{2k-2}$, $0<d<p^2-p$ and $p\nmid d$;
\item[\textup{(c)}]$E(\bar\epsilon_1) \otimes \F_p\{\lambda_2 \mu_2^{-j}\} \otimes
P_{\rho(2k)}(t\mu_2)$
surjectively onto
$E(\bar\epsilon_1) \otimes \F_p\{t^j \lambda_2\} \otimes
P_{\rho(2k-2)}(t\mu_2)$
for each $k\ge2$, $j = dp^{2k-1}$ and $0<d<p$;
\item[\textup{(d)}]
the remaining terms to zero.
\end{itemize}
\end{proposition}

Notice that in the statements (b) and (c) above, we 
abuse notation and indentify the components of degree $>2p-2$
of $E^\infty(S^1, \ell/p)$ and $\mu_2^{-1}E^\infty(S^1,
\ell/p)$, using Theorem~\ref{kkmodp-thm:7.12}(b).

\begin{proof}
Consider the summands of $E^\infty(S^1, \ell/p)$ and $\hat
E^\infty(S^1, \ell/p)$ given in Theo\-rem~\ref{kkmodp-thm:7.12}.  Clearly, the first
term $E(\lambda_2) \otimes \F_p\{\mu_0^i \mid 0<i<p\} \otimes P(\mu_2)$
goes to zero (these classes are hit by $d^2$-differentials), and the
last term $E(\bar\epsilon_1, \lambda_2) \otimes P(t\mu_2)$ maps
identically to the same term.  This proves~(a) and part of~(d).

For each $k\ge1$ and $j = dp^{2k-2}$ with $p\nmid d$, the term
$E(\lambda_2) \otimes \F_p\{\mu_2^{-j}\} \otimes
P_{\rho(2k-1)}(t\mu_2)$ maps to the term $E(\lambda_2) \otimes
\F_p\{t^j\} \otimes P_{\rho(2k-3)}(t\mu_2)$, except that the target is
zero for $k=1$.  In symbols, the element $\lambda_2^\delta \mu_2^{-j}
(t\mu_2)^i$ maps to the element $\lambda_2^\delta t^j (t\mu_2)^{i-j}$.
If $d<0$, then the $t$-exponent in the target is bounded above by
$dp^{2k-2} + \rho(2k-3) < 0$, so the target lives in the right
half-plane and is essentially not hit by the source, which lives in the
left half-plane.  If $d>p^2-p$, then the total degree in the source is
bounded above by $(2p^2-1) - 2dp^{2k} + \rho(2k-1)(2p^2-2) < 2p-2$, so
the source lives in total degree $< (2p-2)$ and will be disregarded.
If $0<d<p^2-p$, then $\rho(2k-1)-dp^{2k-2} > \rho(2k-3)$ and
$-dp^{2k-2} < 0$, so the source surjects onto the target.  This
proves~(b) and part of~(d).

Lastly, for each $k\ge1$ and $j = dp^{2k-1}$ with $p\nmid d$, the term
$E(\bar\epsilon_1) \otimes \F_p\{\lambda_2 \mu_2^{-j}\} \otimes
P_{\rho(2k)}(t\mu_2)$ maps to the term $E(\bar\epsilon_1) \otimes
\F_p\{t^j \lambda_2\} \otimes P_{\rho(2k-2)}(t\mu_2)$.  The target is
zero for $k=1$.  If $d<0$, then $dp^{2k-1} + \rho(2k-2) < 0$ so the
target lives in the right half-plane.  If $d>p$, then $(2p-1) +
(2p^2-1) - 2dp^{2k+1} + \rho(2k)(2p^2-2) < 2p-2$, so the source lives
in total degree $< (2p-2)$.  If $0<d<p$, then $\rho(2k)-dp^{2k-1} >
\rho(2k-2)$ and $-dp^{2k-1} < 0$, so the source surjects onto the
target.  This proves~(c) and the remaining part of~(d).
\end{proof}

\begin{definition}
Let
\begin{equation*}
\begin{aligned}
A &= E(\bar\epsilon_1, \lambda_2) \otimes P(t\mu_2) \\
B_k &= E(\lambda_2) \otimes \F_p\{t^{dp^{2k-2}} \mid 0<d<p^2-p,
	p\nmid d\} \otimes P_{\rho(2k-3)}(t\mu_2) \\
C_k &= E(\bar\epsilon_1) \otimes \F_p\{t^{dp^{2k-1}} \lambda_2 \mid
	0<d<p\} \otimes P_{\rho(2k-2)}(t\mu_2)
\end{aligned}
\end{equation*}
for $k\ge2$ and let $D$ be the span of the remaining monomials in $\hat
E^\infty(S^1, \ell/p)$.  Let $B = \bigoplus_{k\ge2} B_k$ and $C =
\bigoplus_{k\ge2} C_k$.  Then $\hat E^\infty(S^1, \ell/p) = A \oplus B
\oplus C \oplus D$.
\end{definition}

\begin{proposition}\label{kkmodp-prop:8.4}
In degrees $> (2p-2)$, there are closed subgroups $\widetilde A =
E(\bar\epsilon_1, \lambda_2) \otimes P(v_2)$, $\widetilde B_k$, $\widetilde
C_k$ and $\widetilde D$ in $V(1)_* TF(\ell/p;p)$,
represented by the subgroups $A$, $B_k$,
$C_k$ and $D$ of $\hat E^\infty(S^1, \ell/p)$, respectively, such that
the homomorphism $R_* = V(1)_* R$ induced by the restriction map $R$
\begin{itemize}
  \item[\textup{(a)}]is the identity on $\widetilde A$;
  \item[\textup{(b)}]maps $\widetilde B_{k+1}$ surjectively onto $\widetilde B_k$ for all $k\ge2$;
  \item[\textup{(c)}]maps $\widetilde C_{k+1}$ surjectively onto $\widetilde C_k$ for all $k\ge2$;
  \item[\textup{(d)}]is zero on $\widetilde B_2$, $\widetilde C_2$ and $\widetilde D$.
\end{itemize}
In these degrees, $V(1)_* TF(\ell/p;p) \cong \widetilde A \oplus \widetilde B
\oplus \widetilde C \oplus \widetilde D$, where $\widetilde B = \prod_{k\ge2}
\widetilde B_k$ and $\widetilde C = \prod_{k\ge2} \widetilde C_k$.
\end{proposition}

\begin{proof}
The proof is the same as the proof of \cite{kkmodp-AR02}*{Thm.~7.7}, except that
in the present paper we work with the Tate model $\thh(\ell/p)^{tS^1}$
for $TF(\ell/p; p)$, in place of the homotopy fixed point model
$\thh(\ell/p)^{hS^1}$.  The computations are made in $V(1)$-homotopy,
and we disregard all classes in total degrees $\le (2p-2)$.
For example with this convention we write 
$\mu_2^{-1} E^\infty(S^1, \ell/p) \cong E^\infty(S^1,
\ell/p)$, using the same abuse of notation as in
Proposition~\ref{kkmodp-prop:8.2}.

In these terms, the restriction homomorphism $R_*$ is given at the
level of $E^\infty$-terms as the composite of the
isomorphism 
\begin{equation*}
G_* \:  \hat E^\infty(S^1, \ell/p)\to \mu_2^{-1} E^\infty(S^1, \ell/p)
\cong E^\infty(S^1, \ell/p) 
\end{equation*}
and the map
\begin{equation*}
E^\infty(R^h) \: 
E^\infty(S^1, \ell/p) \to \hat E^\infty(S^1, \ell/p)\,.
\end{equation*}
As an endomorphism of $\hat E^\infty(S^1, \ell/p)$, this composite
$E^\infty(R^h)G_*$
is the identity on $A$, maps $B_{k+1}$ onto $B_k$ and
$C_{k+1}$ onto $C_k$ for all $k\ge2$, and is zero on $B_2$, $C_2$
and $D$, by Theorem~\ref{kkmodp-thm:7.12}(c) and
Proposition~\ref{kkmodp-prop:8.2}.
The task is to find closed lifts of these groups to $V(1)_* TF(\ell/p; p)$
such that $R_*$ has similar properties.

Let $\widetilde A = E(\bar\epsilon_1, \lambda_2) \otimes P(v_2) \subset V(1)_*
TF(\ell/p; p)$ be the (degreewise finite, hence closed) subalgebra
generated by the images of the classes $\bar\epsilon_1^K$, $\lambda_2$
and $v_2$ in $V(1)_* K(\ell/p)$.  Then $\widetilde A$ lifts $A$ and consists
of classes in the image of the trace map from $V(1)_* K(\ell/p)$.
Hence $R_*$ is the identity on $\widetilde A$.

We fix $k\ge2$ and choose, for all $n\ge0$, a subgroup
$B^n_k\subset B_{k+n}$, as follows. We take 
\begin{equation*}
  \begin{split}
  B^0_k&=B_k\cap\ker(E^\infty(R^h)G_*) \\
  &=\begin{cases}
  B_2&\textup{ for $k=2$,}\\
  E(\lambda_2) \otimes
  \displaystyle{
  \bigoplus_{0<d<p^2-p,\ p\nmid d}
  }
  \F_p\{t^{dp^{2k-2}}\} \otimes
	 P_{dp^{2k-4}+\rho(2k-5)}^{\rho(2k-3)-1}(t\mu_2)
  &\textup{ for $k\ge3$,}
  \end{cases}
\end{split}
\end{equation*}
where $P_a^b(t\mu_2)=\F_p\{(t\mu_2)^c\,|\,a \le c \le
b\}$. We proceed by induction on $n$ for $n\ge1$, 
choosing a subgroup $B^n_k$ of $B_{k+n}$ mapping
isomorphically onto $B_k^{n-1}$ under $E^\infty(R^h)G_*$
(such a group exists by Theorem~\ref{kkmodp-thm:7.12}(c) and
Proposition~\ref{kkmodp-prop:8.2}(b)). We then have
\begin{equation*}
  B_k=\bigoplus_{n=0}^{k-2}B^n_{k-n}\,.
\end{equation*}
By the argument given on top of page 31 of~\cite{kkmodp-AR02}, we
can choose a lift $\widetilde B_k^0$ of $B_k^0$ with 
\begin{equation*}
  \widetilde B_k^0\subset \im(R_*)\cap \ker(R_*)
\end{equation*}
in $V(1)_* TF(\ell/p; p)$.
By induction on $n\ge1$, we choose a lift
$\widetilde B_k^n\subset \im(R_*)$ of $B_k^n$   
mapping isomorphically onto $\widetilde B_k^{n-1}$ under $R_*$.
Such a choice is possible
since the image of $R_*$ on
$V(1)_*TF(\ell/p;p)$ equals the image of its restriction to
$\im(R_*)$, see~\cite{kkmodp-AR02}*{p.\,30}. 
Now
\begin{equation*}
  \widetilde B_k=\bigoplus_{n=0}^{k-2}\widetilde B_{k-n}^n
\end{equation*}
is a (degreewise finite, hence closed) lift of $B_k$ with $R_*(\widetilde
B_2)=0$
and $R_*(\widetilde B_k)=\widetilde B_{k-1}$ for $k\ge3$. 

To construct $\widetilde C_k$ we proceed as for $\widetilde B_k$ above,
starting with $C_2^0 = C_2$ and
\begin{equation*}
  \begin{split}
  C^0_k&=C_k\cap\ker(E^\infty(R^h)G_*) \\
  &= E(\bar\epsilon_1) \otimes \bigoplus_{0<d<p}
  \F_p\{t^{dp^{2k-1}}\lambda_2\} \otimes
	P_{dp^{2k-3}+\rho(2k-4)}^{\rho(2k-2)-1}(t\mu_2)
\end{split}
\end{equation*}
for $k\ge3$,
and using Theorem~\ref{kkmodp-thm:7.12}(c) and
Proposition~\ref{kkmodp-prop:8.2}(c) to choose $C_k^n$ for
$n\ge1$.

It remains to construct $\widetilde D$.
By Proposition~\ref{kkmodp-prop:8.2}(d), the isomorphism $G_*$ maps $D$ into $\ker
(E^\infty(R^h))$.  By \cite{kkmodp-AR02}*{Lem.~7.3} the representatives in
$E^\infty(S^1, \ell/p)$ of the kernel of $R^h_*$ equal the kernel
of $E^\infty(R^h)$.  It follows that the representatives in $\hat
E^\infty(S^1, \ell/p)$ of the kernel of $R_*$ are mapped isomorphically
by $G_*$ to $\ker(E^\infty(R^h))$.
Hence we can pick a vector space basis for $D$, choose a representative
in $\ker(R_*) \subset V(1)_* TF(\ell/p; p)$ of each basis element, and
let $\widetilde D \subset V(1)_* TF(\ell/p; p)$ be the closure of the vector space
spanned by these chosen representatives.  This closure is contained in
$\ker(R_*)$ since $R_*$ is continuous.  Hence $R_*$ is zero
on $\widetilde D$.
\end{proof}

\begin{proposition}\label{kkmodp-prop:8.6}
In degrees $>(2p-2)$ there are isomorphisms
\begin{equation*}
\begin{aligned}
\ker(R_*-1) &\cong \widetilde A \oplus \lim_k \widetilde B_k
	\oplus \lim_k \widetilde C_k \\
	&\cong E(\bar\epsilon_1, \lambda_2) \otimes P(v_2) \\
&\qquad \oplus E(\lambda_2) \otimes \F_p\{t^d \mid 0<d<p^2-p,
	p\nmid d\} \otimes P(v_2) \\
&\qquad \oplus E(\bar\epsilon_1) \otimes \F_p\{t^{dp} \lambda_2 \mid
	0<d<p\} \otimes P(v_2)
\end{aligned}
\end{equation*}
and $\cok(R_*-1) \cong \widetilde A = E(\bar\epsilon_1, \lambda_2)
\otimes P(v_2)$.  Hence there is an isomorphism
\begin{equation*}
\begin{aligned}
V(1)_* TC(\ell/p;p) &\cong
	E(\partial, \bar\epsilon_1, \lambda_2) \otimes P(v_2) \\
&\qquad \oplus E(\lambda_2) \otimes \F_p\{t^d \mid 0<d<p^2-p,
	p\nmid d\} \otimes P(v_2) \\
&\qquad \oplus E(\bar\epsilon_1) \otimes \F_p\{t^{dp} \lambda_2 \mid
	0<d<p\} \otimes P(v_2)
\end{aligned}
\end{equation*}
in these degrees, where $\partial$ has degree $-1$ and represents the
image of $1$ under the connecting map $\partial$
in~\eqref{kkmodp-eq:8.1}.
\end{proposition}

\begin{proof}
By Proposition~\ref{kkmodp-prop:8.4}, the homomorphism $R_*-1$ is zero on $\widetilde A$
and an isomorphism on $\widetilde D$.  Furthermore, there is an exact
sequence
\begin{equation*}
0 \to \lim_k \widetilde B_k \to \prod_{k\ge2} \widetilde B_k
\xr{R_*-1} \prod_{k\ge2} \widetilde B_k \to \Rlim_k \widetilde B_k \to 0
\end{equation*}
and similarly for the $C$'s.  The derived limit on the right vanishes
since each $\widetilde B_{k+1}$ surjects onto $\widetilde B_k$.

Multiplication by $t\mu_2$ in each $B_k$ is realized by multiplication by
$v_2$ in $\widetilde B_k$.  Each $\widetilde B_k$ is a sum of $2(p-1)^2$
cyclic $P(v_2)$-modules, and since $\rho(2k-3)$ grows to infinity with $k$
their limit is a free $P(v_2)$-module of the same rank, with the indicated
generators $t^d$ and $t^d \lambda_2$ for $0 < d < p^2-p$, $p \nmid d$.
The argument for the $C$'s is practically the same.

The long exact sequence~\eqref{kkmodp-eq:8.1} yields the short exact sequence
\begin{equation*}
0 \to \Sigma^{-1} \cok(R_*-1) \xr{\partial} V(1)_* TC(\ell/p;p) \xr{\pi}
\ker(R_*-1) \to 0 \,,
\end{equation*}
from which the formula for the middle term follows.
\end{proof}

\begin{remark}
A more obvious set of $E(\lambda_2) \otimes P(v_2)$-module generators
for $\lim_k \widetilde B_k$ would be the classes $t^{dp^2}$ in $B_2
\cong \widetilde B_2$, for $0 < d < p^2-p$, $p \nmid d$.  
We have a commutative diagram
\begin{equation*}
  \xymatrix{
  TF(\ell/p;p)\ar[r]^-{\hat\Gamma} \ar[d] &
   \thh(\ell/p)^{tS^1} \ar[d]^{FG} \\
   \thh(\ell/p)^{C_p} \ar[r]^-{G_1\hat\Gamma_2}  &
   (\rho_p^*\thh(\ell/p)^{tC_p})^{hC_p}\,.
  }
\end{equation*}
Under the left-hand
canonical map $TF(\ell/p;p) \to \thh(\ell/p)^{C_p}$, modeled here by
\begin{equation*}
FG \: \thh(\ell/p)^{tS^1} \to
(\rho_p^*\thh(\ell/p)^{tC_p})^{hS^1}\to
(\rho_p^*\thh(\ell/p)^{tC_p})^{hC_p} \,,
\end{equation*}
the class $t^{dp^2}$ maps to
the class $\mu_2^{-d}$.  Since we are only concerned with degrees
$> (2p-2)$ we may equally well use its $v_2$-power
multiple
$(t\mu_2)^d \cdot \mu_2^{-d} = t^d$ as generator, with the advantage
that it is in the image of the localization map 
\begin{equation*}
\thh(\ell/p)^{hC_p}
\to (\rho_p^*\thh(\ell/p)^{tC_p})^{hC_p}\,.
\end{equation*}
Hence the class denoted $t^d$ in
$\lim_k \widetilde B_k$ is chosen so as to map under
$TF(\ell/p;p) \to
\thh(\ell/p)^{hC_p}$ to $t^d$ in $E^\infty(C_p, \ell/p)$.  Similarly,
the class denoted $t^{dp} \lambda_2$ in $\lim_k \widetilde C_k$ is chosen
so as to map to $t^{dp} \lambda_2$ in $E^\infty(C_p, \ell/p)$.
\end{remark}

\medskip

The map $\pi \: \ell/p \to \Z/p$ is $(2p-2)$-connected, hence induces
$(2p-1)$-connected maps $\pi_* \: K(\ell/p) \to K(\Z/p)$ and $\pi_*  \: 
V(1)_* TC(\ell/p;p) \to V(1)_* TC(\Z/p;p)$, by
\cite{kkmodp-BM94}*{Prop.~10.9} and
\cite{kkmodp-Du97}*{p.~224}.  Here $TC(\Z/p;p) \simeq H\Z_p \vee \Sigma^{-1} H\Z_p$ and
we have an isomorphism 
$V(1)_* TC(\Z/p;p) \cong E(\partial, \bar\epsilon_1)$, so we can recover
$V(1)_* TC(\ell/p;p)$ in degrees $\le (2p-2)$ from this map.

\begin{theorem}\label{kkmodp-thm:8.8}
There is an isomorphism of $P(v_2)$-modules
\begin{equation*}
\begin{aligned}
V(1)_* TC(\ell/p;p) &\cong P(v_2) \otimes E(\partial, \bar\epsilon_1,
	\lambda_2) \\
&\qquad \oplus P(v_2) \otimes E(\dlog v_1) \otimes \F_p\{t^d
	v_2 \mid 0<d<p^2-p, p\nmid d\} \\
&\qquad \oplus P(v_2) \otimes E(\bar\epsilon_1) \otimes \F_p\{t^{dp}
	\lambda_2 \mid 0<d<p\}
\end{aligned}
\end{equation*}
where $\dlog v_1 \cdot t^dv_2 =t^d \lambda_2$.  
The degrees are $|\partial| =
-1$, $|\bar\epsilon_1| = |\lambda_1| = 2p-1$, $|\lambda_2| =
2p^2-1$, $|v_2| = 2p^2-2$, $|t| = -2$ and $|\dlog v_1| = 1$.
\end{theorem}

The notation $\dlog v_1$ for the multiplier $v_2^{-1} \lambda_2$ is
suggested by the relation $v_1 \cdot \dlog p = \lambda_1$ in $V(0)_*
TC(\Z_{(p)}|\Q;p)$.

\begin{proof}
Only the additive generators $t^d$ for $0<d<p^2-p$, $p\nmid d$ from
Proposition~\ref{kkmodp-prop:8.6} do not appear in $V(1)_*
TC(\ell/p;p)$, but their
multiples by $\lambda_2$ and positive powers of $v_2$ do.  This leads
to the given formula, where $\dlog v_1 \cdot t^d v_2$ must be read as
$t^d\lambda_2$.
\end{proof}

By \cite{kkmodp-HM97}*{Th.~C} the cyclotomic trace map of \cite{kkmodp-BHM93} induces
cofiber sequences
\begin{equation*}
K(B_p)_p \xr{trc} TC(B;p)_p \xr{g} \Sigma^{-1} H\Z_p \to
\Sigma K(B_p)_p 
\end{equation*}
for each connective $S$-algebra $B$ with $\pi_0(B_p) = \Z_p$ or $\Z/p$,
and thus long exact sequences
\begin{equation*}
\dots \to V(1)_* K(B_p) \xr{trc} V(1)_* TC(B;p) \xr{g} \Sigma^{-1}
E(\bar\epsilon_1) \to \dots \,.
\end{equation*}
This uses the identifications $W(\Z_p)_F \cong W(\Z/p)_F \cong \Z_p$ of
Frobenius coinvariants of rings of Witt vectors, and applies in particular for $B
= H\Z_{(p)}$, $H\Z/p$, $\ell$ and $\ell/p$.

\begin{theorem}\label{kkmodp-thm:8.10}
There is an isomorphism of $P(v_2)$-modules
\begin{equation*}
\begin{aligned}
V(1)_* K(\ell/p) &\cong P(v_2) \otimes E(\bar\epsilon_1) \otimes
	\F_p\{1, \partial\lambda_2, \lambda_2, \partial v_2\} \\
&\qquad \oplus P(v_2) \otimes E(\dlog v_1) \otimes \F_p\{t^d
	v_2 \mid 0<d<p^2-p, p\nmid d\} \\
&\qquad \oplus P(v_2) \otimes E(\bar\epsilon_1) \otimes \F_p\{t^{dp}
	\lambda_2 \mid 0<d<p\} \,.
\end{aligned}
\end{equation*}
This is a free $P(v_2)$-module of rank~$(2p^2-2p+8)$ and of zero Euler
characteristic.
\end{theorem}

\begin{proof}
In the case $B = \Z/p$, $K(\Z/p)_p \simeq H\Z_p$ and the map $g$ is
split surjective up to homotopy.  So the induced homomorphism to $V(1)_*
\Sigma^{-1} H\Z_p = \Sigma^{-1} E(\bar\epsilon_1)$ is surjective.  Since
$\pi \: \ell/p \to \Z/p$ induces a $(2p-1)$-connected map in topological
cyclic homology, and $\Sigma^{-1} E(\bar\epsilon_1)$ is concentrated
in degrees $\le (2p-2)$, it follows by naturality that also in the
case $B = \ell/p$ the map $g$ induces a surjection in $V(1)$-homotopy.
The kernel of the surjection $P(v_2) \otimes E(\partial, \bar\epsilon_1,
\lambda_2) \to \Sigma^{-1} E(\bar\epsilon_1)$ gives the first row in
the asserted formula.
\end{proof}

\begin{acknowledgements}
 We would like to thank the referee for
 his many useful suggestions. Both authors were
 partially supported by the Max Planck Institute for
 Mathematics, Bonn.
\end{acknowledgements}

\end{document}